\documentclass[a4 paper, 12pt]{article}

\usepackage{amsmath,mathrsfs,pictexwd,dcpic}
\usepackage{amssymb}
\usepackage[left = 3cm,right=2cm,top=3cm]{geometry}
\usepackage{amsthm}
\usepackage{caption}
\usepackage{graphicx}
\usepackage{newlfont,amsfonts}
\usepackage{ytableau}
\usepackage{enumitem}
\usepackage{setspace}
\usepackage{hyperref}
\theoremstyle{plain}

\newtheorem{theorem}{Theorem}
\newtheorem{lemma}[theorem]{Lemma}
\newtheorem{corollary}[theorem]{Corollary}
\theoremstyle{definition}
\newtheorem{definition}[theorem]{Definition}
\newtheorem{example}[theorem]{Example}
\theoremstyle{remark}
\newtheorem{remark}[theorem]{Remark}
\newtheorem{notation}[theorem]{Notation}

\newcommand{\seqnum}[1]{\href{https://oeis.org/#1}{\rm \underline{#1}}}
\allowdisplaybreaks
\title{$p^{(k)}$-Fibonacci Numbers of the $p$-Bratteli Diagram for Every Odd Prime $p$ and Integer $k\geq0$}
\author{M. Parvathi, A. Tamilselvi, and D. Hepsi\\
	Ramanujan Institute for Advanced Study in Mathematics\\ University of Madras \\ Chennai-600 005 \\ India\\ tamilselvi.riasm@gmail.com }
\date{}
\begin{document}
	\maketitle
	\begin{abstract}
		We study paths in the $p$-Bratteli diagram associated with hook partitions, where $p$ is an odd prime. By comparing blocks along a path, we define inversions and descents. We prove that the sign balance derived from inversions vanishes at every vertex of the diagram. Using descents, we introduce the $p^{(k)}$-Fibonacci numbers and derive recurrence relations for them. For $k=0$, we recover the OEIS sequence \seqnum{A391520}, while for $k \ge 1$ we obtain new families of Fibonacci-type sequences.
	\end{abstract}
	\section{Introduction}
	
	Researchers first introduced Bratteli diagrams to study finite-dimensional $C^*$-algebras, and later, mathematicians adopted them as an important combinatorial tool in the representation theory of groups and algebras. When vertices are indexed by partitions, Bratteli diagrams provide a clear and computable description of how representations evolve from one level to the next.
	
	In this paper, we focus on the $p$-Bratteli diagram introduced in~\cite{[TH],[TH1]} for an odd prime $p$. Its even levels consist of vertices indexed by hook partitions corresponding to a complete set of inequivalent irreducible representations of the group algebras $KG_r$, while its odd levels consist of vertices indexed by hook partitions corresponding to a complete set of inequivalent irreducible representations of certain subalgebras $KSG_r$. Apart from its representation-theoretic importance, the $p$-Bratteli diagram also has a rich and regular combinatorial structure.
	
	Fibonacci numbers are classical integer sequences introduced by Leonardo of Pisa in the \emph{Liber Abaci}. They appear in many areas of mathematics, and several generalizations such as $k$-Fibonacci numbers and $q$-Fibonacci numbers have been studied extensively; see, for example,~\cite{[AA],[GS],[MA],[SS]}. These generalizations often reflect deeper combinatorial or algebraic structures.
	
	The main motivation of this paper is that Fibonacci-type behavior arises naturally from the path structure of the $p$-Bratteli diagram. We study paths whose vertices are indexed by hook partitions and describe each path as a sequence of blocks added to a partition. By comparing the horizontal and vertical parts of blocks of the same size, we define \emph{inversions} and \emph{descents} for such paths. These notions are natural analogues of classical permutation statistics \cite[p.\ 221]{[BS]}.
	
	In our study of inversions, we assign a sign to each path according to the number of inversions and define the \emph{sign balance} at a vertex as the sum of the signs of all paths ending at that vertex. We prove that this sign balance is zero at every vertex of the $p$-Bratteli diagram, showing a strong cancellation phenomenon.
	
	Our main result is the introduction of a new family of integers, called the \emph{$p^{(k)}$-Fibonacci numbers}. For $s \ge 1$, we define these numbers as the total number of descents over all paths ending at a fixed vertex
	\[
	\lambda\big(2r;(x_{(2s,k)},x_{(s,k)}+l)\big).
	\]
	We denote this number by $\mathcal{M}\Big(\lambda\big(2r;(x_{(2s,k)},x_{(s,k)}+l)\big)\Big)$. We justify the terminology $p^{(k)}$-Fibonacci numbers in Theorem \ref{rr}. By studying the local structure of the $p$-Bratteli diagram,  we show that these numbers satisfy Fibonacci-type recurrence relations.

	In the case $k=0$, we obtain a closed formula and identify the resulting sequence in Theorems~\ref{p0} listed as \seqnum{A391520} in the On-Line Encyclopedia of Integer Sequences. For $k \ge 1$, we obtain several infinite families of $p$-tuples of sequences in Theorems~\ref{p1}, \ref{base}, \ref{less}, and \ref{greater}, that, to the best of our knowledge, are new. This provides a systematic link between descent statistics on the $p$-Bratteli diagram and Fibonacci-type sequences. We believe that these sequences will lead to further applications and developments in future studies.
	
	We make the paper self-contained, with the necessary background recalled where needed and organized as follows. In Section~$2$, we introduce notation and basic definitions. Section~$3$ describes the construction of the $p$-Bratteli diagram. Section~$4$ introduces inversions and descents and proves the vanishing of the sign balance. Section~$5$ studies the $p^{(k)}$-Fibonacci numbers, their recurrence relations,  and examples. In Section~$6$, we discuss the generating functions for the sequence of $p^{(k)}$-Fibonacci numbers  $\big(\mathcal{M}\big(\lambda\big(2r;(x_{(2s,k)},x_{(s,k)}+l)\big)\big)\big)^{\infty}_{s= k+2}$ are discussed for every odd prime $p$ and every $k\geq0$. 
	
	\section{Notation and definitions}
	\begin{notation}\label{not}
		\item
		\begin{enumerate}
			\item Throughout this paper, $p$ denotes an odd prime number. The symbols 
			$$a,~b,~i,~i^{'},~j,~k,~l,~l^{'},~l_{1},~l_{2},~m,~n,~t,~t_{1},~t_{2},~t^{'},~t^{''}, \alpha,~\beta,~\iota,~\mu,~\theta$$  denote non-negative integers. The symbols $r,~s$ denote positive integers. The symbols $m_{i}~\text{and}~n_{i}$ denote non-negative integers where $i=1,2,\ldots,2r-1$.
			\item $j^{k}_{t}=t\sum\limits_{i=0}^{k-1}p^{i}$, where $0\leq t \leq p-1$ and $k\geq 1$.	
			\item $[a, a+1\ldots,b]$ denotes the non-negative integers from $a$ to $b$. 
			\item $x_{(a,k)}=p^{k}(ap-(a+1))$, $a\geq 1$ and $k\geq 0$. \label{x_{a,k}}
		\end{enumerate}
		\end{notation}
	\begin{definition} \cite[p.\ 2]{[BS]}
		Let $\lambda_{1}$, $\lambda_{2}$, $\ldots$, $\lambda_{l}$ be non-negative integers. A partition of $n$ is a sequence $\lambda=(\lambda_{1},\lambda_{2},\ldots,\lambda_{l})$ where $\lambda_{i}$ are weakly decreasing and $\sum_{i=1}^{l}\lambda_{i}=n$.
	\end{definition}
	\begin{definition} \cite[p. \ 58]{[BS]}
		Let $\lambda_{1}$, $\lambda_{2}$, $\ldots$, $\lambda_{l}$ and $\mu_{1}$, $\mu_{2}$, $\ldots$, $\mu_{m}$ be non-negative integers. Suppose $\lambda=(\lambda_{1},\lambda_{2},\ldots,\lambda_{l})$ and $\mu=(\mu_{1},\mu_{2},\ldots,\mu_{m})$ are partitions of $n$. Then $\lambda$ dominates $\mu$, written $\lambda \trianglerighteq \mu$, if $$\lambda_{1}+\lambda_{2}+\cdots+\lambda_{i}\geq \mu_{1}+\mu_{2}+\cdots+\mu_{i} $$ for all $i \geq 1$. If $i>l$ (respectively, $i>m$), then we take $\lambda_{i}$ (respectively $\mu_{i}$) to be zero.
	\end{definition}	
	\begin{definition}
A partition $\lambda$ of $n$ is called a hook partition if $\lambda=(n-i,1^{i})$ for some $0 \leq i < n$.
If such a hook partition appears on the $c$-th floor, we denote it by $\lambda(c;(n,i))$.
The corresponding Young diagram consists of $n-i$ horizontal nodes and $i$ vertical nodes.
	\end{definition}
	\begin{notation}
		\item
		\begin{enumerate}
		\item $B(i;(m,n))$ denotes the $i$-th block in a hook partition $\lambda$ of size $m+n$ which contains $m$ nodes horizontally and $n$ nodes vertically, where by vertical nodes we mean the nodes in the first column other than the first row. In particular, $B(i;(0,0))$ denotes the empty block.
	\item $\lambda\overset{B(i;(m,n))}{\hookrightarrow}\mu$ denotes the partition $\mu$ obtained by adding the $i$-th block $B(i;(m,n))$ to $\lambda$.	
	\item $\lambda\overset{B(i;(m,n))}{\hookleftarrow}\mu$ denotes the partition $\mu$ obtained by removing the $i$-th block $B(i;(m,n))$ from $\lambda$ .
			\end{enumerate}
		\end{notation}
	\section{The \texorpdfstring{$p$}{p}-Bratteli diagram}
	In this section, we define the vertices and edges of the $p$-Bratteli diagram in an alternative form from \cite{[TH1]}. The hook partitions correspond to the vertices. We also explain the ordering and arrangement of these vertices within the $p$-Bratteli diagram. We discuss each path in the $p$-Bratteli diagram in detail, and introduce a simplified notation for representing such paths. Moreover, each path in the $p$-Bratteli diagram corresponds to a standard $p$-Young tableau, as defined in \cite[Thm. \ 3.4, 3.5]{[TH1]}. 
	
	For $r\geq 1,$ the $p$-Bratteli diagram consists of vertices in the sets $V^{2r}$ in the $2r$-th floor and $W^{2r-1}$ in the $(2r-1)$-th floor together with the corresponding edges defined below.
	\subsection{Vertex set of the \texorpdfstring{$p$}{p}-Bratteli diagram}\label{vertex}
	\subsubsection*{Vertices on the $(2r-1)$-th floor for $r\geq 1$:}
	$$W^{2r-1}= \bigcup_{k=0}^{r-1} W^{2r-1}_{k} $$
	where
	\begin{itemize}
		\item $W_{r-1}^{2r-1} = \{\lambda\big(2r-1;(p^{r-1}(p-1),i)\big)|0\leq i<p^{r-1}(p-1)\}$. 
		\item 	
		$W^{2r-1}_{k} =
		\left\{
		\lambda\big(2r-1;(x_{(2(r-k)-1,k)},\,
			x_{(r-k-1,k)}+l^{'})\big) 
		~\big{|}
		0 \le l^{'} < p^{k+1}
		\right\},~|W^{2r-1}_{k}|=p^{k+1}$, for all $0\leq k \leq r-2,\; r \ge 2$, where $l^{'}$ can also be written as $l^{'}=l+p^{k}t$, with $0\leq l<p^{k}$ and $0\leq t\leq p-1$.
	\end{itemize}
	
	The vertex subsets are arranged in the following order: $$W^{2r-1}_{r-1},~W^{2r-1}_{r-2},~W^{2r-1}_{r-3},\ldots,~W^{2r-1}_{1},~W^{2r-1}_{0}.$$
	
	The vertices in each vertex subset $W^{2r-1}_{k}$ are ordered according to the dominance order, where $l^{'}$ corresponds to the position of the vertex.
	\subsubsection*{Vertices on the $2r$-th floor for $r\geq 1$:}
	$$V^{2r}=\bigcup_{k=0}^{r}V^{2r}_{k}$$
	where
	\begin{itemize}
		\item $V^{2r}_{r} = \{\lambda\big(2r;(p^{r-1}(p-1),i)\big)|0\leq i<p^{r-1}(p-1)\}$. 
		
		\item $
		V^{2r}_{k} =
		\left\{
		\begin{array}{l}
			\lambda\big(2r;(x_{(2(r-k),k)},\,
				x_{(r-k,k)}+l)\big) ~\big{|}
			0 \le l < p^{k}
		\end{array}
		\right\},~|V^{2r}_{k}|=p^{k}$, for all $0\leq k \leq r-1$.
	\end{itemize}
	
	The vertex subsets are arranged in the following order: $$V^{2r}_{r},~V^{2r}_{r-1},~V^{2r}_{r-2},\ldots,~V^{2r}_{1},~V^{2r}_{0}.$$
	The vertices in each vertex subset $V^{2r}_{k}$ are ordered according to the dominance order, where $l$ corresponds to the position of the vertex.
	\begin{remark}\cite{[TH1]}
		\begin{enumerate}
			\item For $r \geq 1$, the elements of $V_{r}^{2r}~\big(\text{or}~W_{r-1}^{2r-1}\big)$ correspond to the indexing sets for all one dimensional representations of the class of group algebras $KG_{r}$ (or subalgebras $KSG_{r}$).
			\item The elements of $V_{k}^{2r}~\big(\text{or}~W_{k}^{2r-1}\big)$ with $0\leq k \leq r-1$ (or $0\leq k\leq r-2$) correspond to the indexing set for the complete set of inequivalent irreducible representations of dimension $p^{r-k-1}(p-1)~\big(\text{or}~p^{r-k-2}(p-1)\big)$ of the class of group algebras $KG_{r}$ (or subalgebras $KSG_{r}$), for $r\ge 1~\left(\text{or} ~r\ge 2\right)$, respectively.
		\end{enumerate}
	\end{remark}
	\subsection{Edge Set of the \texorpdfstring{$p$}{p}-Bratteli Diagram}\label{edge}
	We define the edges between two vertices of two consecutive floors of the $p$-Bratteli diagram as:
	\subsubsection*{Edges from the $(2r-2)$-th floor to the $(2r-1)$-th floor:}
	\begin{itemize}
		\item Let $\lambda\big(2r-2;(p^{r-2}(p-1),i)\big)\in V^{2r-2}_{r-1}$ and  $\lambda\big(2r-1;(p^{r-1}(p-1),pi+t)\big)\in W^{2r-1}_{r-1}$, then for all $0\leq t\leq p-1$, there is an edge between the vertices by adding the block $B(2r-2;(a_{t},b_{t}))$, as shown in Figure \ref{1}
		\begin{figure}[!ht]
			\centering
			\includegraphics[width=0.7\linewidth]{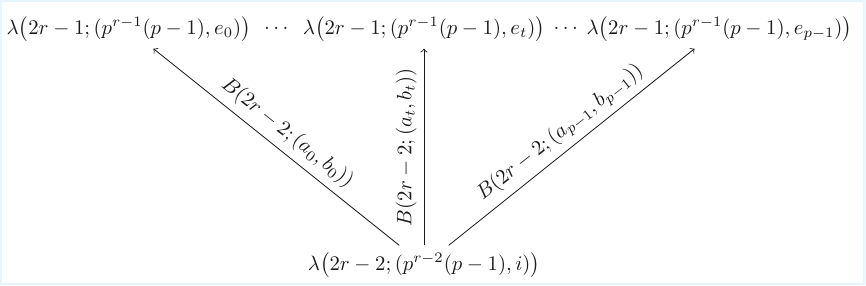}
			\caption{}
			\label{1}
		\end{figure}
		
		where $e_{t}=pi+t$, $a_{t}=p^{r-2}(p-1)^{2}-(i(p-1)+t)$, $b_{t}=i(p-1)+t$.
		$$\lambda\big(2r-2;(p^{r-2}(p-1),i)\big)\overset{B(2r-2;(a_{t},b_{t}))}{\hookrightarrow} \lambda\big(2r-1;(p^{r-1}(p-1),pi+t)\big)$$		
			\item Let $\lambda\big(2r-2;(x_{(2(s-1),k)},x_{(s-1,k)}+l)\big)\in V^{2(r-1)}_{k}$ and $\lambda\big(2r-1;(x_{(2s-1,k)},x_{(s-1,k)}+(pl+t))\big)\in W^{2r-1}_{k}$, then for all $0\leq t\leq p-1$, there is an edge between the vertices by adding the block $B(2r-2;(a_{t},b_{t}))$, as shown in Figure \ref{2},
		\begin{figure}[!ht]
			\centering
			\includegraphics[width=0.7\linewidth]{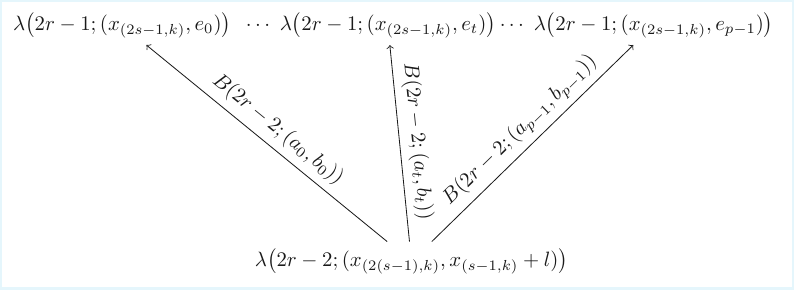}
			\caption{}
			\label{2}
		\end{figure}
		
			where $s=r-k$, $e_{t}=x_{(s-1,k)}+(pl+t)$, $a_{t}=p^{k}(p-1)-((p-1)l+t)$, $b_{t}=(p-1)l+t$, $0\leq l <p^{k}$, and $x_{(a,k)}$ is as in Notation \ref{not}(\ref{x_{a,k}}). 
		$$\lambda\big(2r-2;(x_{(2(s-1),k)},x_{(s-1,k)}+l)\big)\overset{B(2r-2;(a_{t},b_{t}))}{\hookrightarrow} \lambda\big(2r-1;(x_{(2s-1,k)},x_{(s-1,k)}+(pl+t))\big)$$
	
	\end{itemize}
	\subsubsection*{Edges from the $(2r-1)$-th floor to the $2r$-th floor:}
	\begin{itemize}
		\item Let $\lambda\big(2r-1;(p^{r-1}(p-1),i)\big)\in W^{2r-1}_{r-1}$ and $\lambda\big(2r;(p^{r-1}(p-1),i)\big)\in V^{2r}_{r}$, then there is an edge between the vertices by adding the empty block $B(2r-1;(0,0))$, $$\lambda\big(2r-1;(p^{r-1}(p-1),i)\big)\overset{B(2r-1;(0,0))}{\hookrightarrow}\lambda\big(2r;(p^{r-1}(p-1),i)\big)$$ 
		\item Let $\lambda\big(2r-1;(p^{r-1}(p-1),i)\big)\in W^{2r-1}_{r-1}$ with $p^{r-1}t\leq i<p^{r-1}(t+1)$, where $0\leq t\leq p-2$, and let $\lambda\big(2r;(x_{(2,r-1)},i+p^{r-1}(p-2-t))\big)\in V^{2r}_{r-1}$. 
		Then there is an edge between the vertices obtained by adding the block $B(2r-1;(p^{r-1}t,p^{r-1}(p-2-t)))$,
		$$\lambda\big(2r-1;(p^{r-1}(p-1),i)\big)  \overset{B(2r-1;(p^{r-1}t,p^{r-1}(p-2-t)))}{\hookrightarrow}\lambda\big(2r;(x_{(2,r-1)},i+p^{r-1}(p-2-t))\big).$$
		\item Let $\lambda\big(2r-1;(x_{(2s-1,k)},x_{(s-1,k)}+(l+p^{k}t))\big)\in W^{2r-1}_{k}$ and $\lambda\big(2r;(x_{(2s,k)},x_{(s,k)}+l)\big) \in V^{2r}_{k}$. Then for $0\leq t\leq p-1$,  there is an edge between the vertices obtained by adding the block $B(2r-1;(p^{k}t,p^{k}(p-1-t)))$,  $$\lambda\big(2r-1;(x_{(2s-1,k)},x_{(s-1,k)}+(l+p^{k}t))\big)\overset{B(2r-1;(p^{k}t,p^{k}(p-1-t)))}{\hookrightarrow}\lambda\big(2r;(x_{(2s,k)},x_{(s,k)}+l)\big)$$ 
		where $s=r-k$, $0\leq l<p^{k}$, and $x_{(a,k)}$ is as in Notation \ref{not}(\ref{x_{a,k}}).
	\end{itemize}
	\subsubsection*{Edges from the $2r$-th floor to the $(2r-1)$-th floor:}
	\begin{itemize}
		\item Let $\lambda\big(2r;(p^{r-1}(p-1),i)\big)\in V^{2r}_{r}$ and $\lambda\big(2r-1;(p^{r-1}(p-1),i)\big)\in W^{2r-1}_{r-1}$, then there is an edge between the vertices by removing the empty block $B(2r-1;(0,0))$, $$\lambda\big(2r;(p^{r-1}(p-1),i)\big)\overset{B(2r-1;(0,0))}{\hookleftarrow}\lambda\big(2r-1;(p^{r-1}(p-1),i)\big)$$ 
		\item  Let $\lambda\big(2r;(x_{(2,r-1)},x_{(1,r-1)}+l)\big) \in V^{2r}_{r-1}$ and $\lambda\big(2r-1;(p^{r-1}(p-1),x_{(1,r-1)}+l-p^{r-1}(p-2-t))\big)\in W^{2r-1}_{r-1}$, then for $0\leq t\leq p-2$ there is an edge between the vertices obtained by removing the block $B(2r-1;(a_{t},b_{t}))$, as shown in Figure \ref{3},
		\begin{figure}[!ht]
			\centering
			\includegraphics[width=0.7\linewidth]{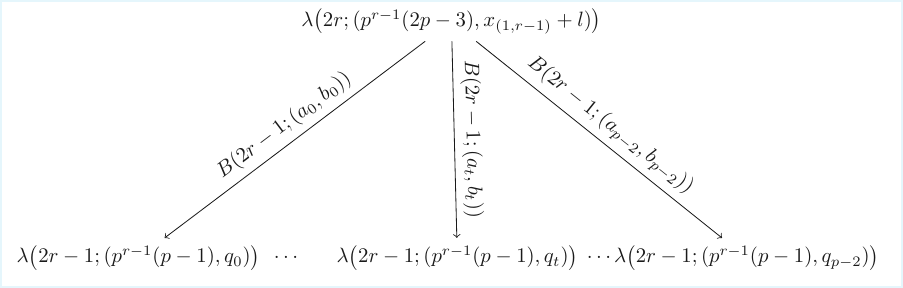}
			\caption{}
			\label{3}
		\end{figure}
				
		where $q_{t}=x_{(1,r-1)}+l-p^{r-1}(p-2-t)$, $a_{t}=p^{r-1}t,~b_{t}=p^{r-1}(p-2-t)$, $0\leq l<p^{k}$, and  $x_{(1,r-1)}=p^{r-1}(p-2)$.
		\begin{multline*}
				\lambda\big(2r;(x_{(2,r-1)},x_{(1,r-1)}+l)\big)\overset{B(2r-1;(a_{t},b_{t}))}{\hookleftarrow} \\ \lambda\big(2r-1;(p^{r-1}(p-1),x_{(1,r-1)}+l-p^{r-1}(p-2-t))\big)
		\end{multline*}		
		\item Let $\lambda\big(2r;(x_{(2s,k)},x_{(s,k)}+l)\big)\in V^{2r}_{k}$ and $\lambda\big(2r-1;(x_{(2s-1,k)},x_{(s-1,k)}+(l+p^{k}t))\big)\in W^{2r-1}_{k}$, then for $0\leq t\leq p-1$ there is an edge between the vertices obtained by removing the block $B(2r;(a_{t},b_{t}))$, as shown in Figure \ref{4}, 
		\begin{figure}[!ht]
			\centering
			\includegraphics[width=0.7\linewidth]{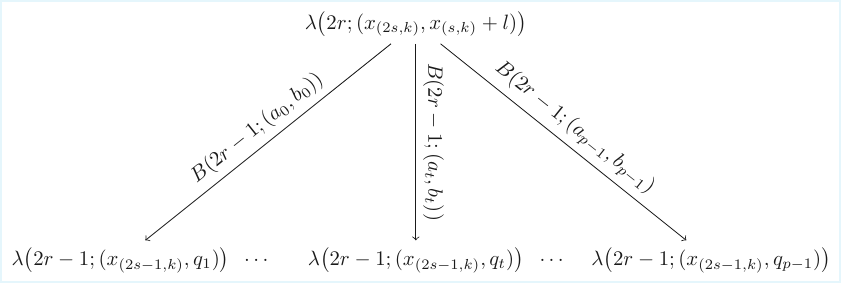}
			\caption{}
			\label{4}
		\end{figure}
		
			where $s=r-k$, $q_{t}=x_{(s-1,k)}+(l+p^{k}t)$, $a_{t}=p^{k}t$, $b_{t}=p^{k}(p-1-t)$, $0\leq l <p^{k}$ and $x_{(a,k)}$ is as in Notation \ref{not}(\ref{x_{a,k}}).
		$$\lambda\big(2r;(x_{(2s,k)},x_{(s,k)}+l)\big)\overset{B(2r;(a_{t},b_{t}))}{\hookleftarrow}\lambda\big(2r-1;(x_{(2s-1,k)},x_{(s-1,k)}+(l+p^{k}t))\big)$$
	
	\end{itemize}
	\subsubsection*{Edges from the $(2r-1)$-th floor to the $2(r-1)$-th floor:}
	\begin{itemize}
		\item Let $\lambda\big(2r-1;(p^{r-1}(p-1),i)\big)\in W^{2r-1}_{r-1}$ with $\bar{i}p\leq i <(\bar{i}+1)p$, where $0\leq \bar{i} < p^{r-2}(p-1)$ and let $\lambda\big(2r-2;(p^{r-2}(p-1),\bar{i})\big)\in V^{2r-2}_{r-1}$. Then there is an edge between the vertices obtained by removing the block $B(2r-2;(p^{r-2}(p-1)^{2}-(i-\bar{i}),(i-\bar{i})))$,  $$\lambda\big(2r-1;(p^{r-1}(p-1),i)\big)\overset{B(2r-2;(p^{r-2}(p-1)^{2}-(i-\bar{i}),(i-\bar{i})))}{\hookleftarrow}\lambda\big(2r-2;(p^{r-2}(p-1),\bar{i})\big).$$
		\item Let $\lambda\big(2r-1;(x_{(2s-1,k)},x_{(s-1,k)}+(l+p^{k}t))\big)\in W^{2r-1}_{k}$ and $\lambda\big(2r-2;(x_{(2(s-1),k)},x_{(s-1,k)}+(\alpha+p^{k-1}t))\big)\in V^{2r-2}_{k}$, then there is exactly one edge from the vertex $\lambda\big(2r-1;(x_{(2s-1,k)},x_{(s-1,k)}+(l+p^{k}t))\big)$ to the vertex $\lambda\big(2r-2;(x_{(2(s-1),k)},x_{(s-1,k)}+(\alpha+p^{k-1}t))\big)$ by removing the block as given below:
		\begin{multline*}
			\lambda\big(2r-1;(x_{(2s-1,k)},x_{(s-1,k)}+(l+p^{k}t))\big) \\	\overset{B(2r-2;(p^{k}(p-1)-(p-1)(\alpha+p^{k-1}t)-\beta,(p-1)(\alpha+p^{k-1}t)+\beta))}{\hookleftarrow} \\ \lambda\big(2r-2;(x_{(2(s-1),k)},x_{(s-1,k)}+(\alpha+p^{k-1}t))\big)  
		\end{multline*}
		Here $s=r-k$, $0\leq l <p^{k}$ and $l= \alpha p + \beta$, where $0\leq \alpha < p^{k-1}$, $0\leq \beta \leq p-1,~0\leq t\leq p-1$ and $x_{(a,k)}$ is as in Notation \ref{not}(\ref{x_{a,k}}).
	\end{itemize}
	\begin{remark}\label{Proj}
		For $0\leq t\leq p-1,$ the vertex $\lambda\big(2r-2;(x_{(2(s-1),k)},x_{(s-1,k)}+(\alpha+p^{k-1}t))\big)$ on the $(2r-2)$-th floor is called the \textit{projection} of the vertex $\lambda\big(2r-1;(x_{(2s-1,k)},x_{(s-1,k)}+(l+p^{k}t))\big)$ on the $(2r-1)$-th floor. The integer $\alpha$ is called the \textit{projection} of the integer $l$.
	\end{remark}		
	\begin{remark}\label{syt def}
		\item
		\begin{enumerate}
			\item In this paper, we focus only on the descents defined between two adjacent blocks of the same size. But every path is obtained by adding the blocks of different sizes $p^{j-1}(p-1)^{2},~j\geq 1$, as explained below
			\begin{multline*}
				\lambda\big(1;(p-1,i_{1})\big)\overset{B(1;(0,0))}{\hookrightarrow	}\lambda\big(2;(p-1,i_{1})\big)\overset{B(2;(m_{2},n_{2}))}{\hookrightarrow}\lambda\big(3;(p(p-1),i_{2})\big)\\ \overset{B(3;(0,0))}{\hookrightarrow}  \lambda\big(4;(p(p-1),i_{2})\big)\overset{B(4;(m_{4},n_{4}))}{\hookrightarrow}\lambda\big(5;(p^{2}(p-1),i_{3})\big)\cdots\\ \overset{B(2k;(m_{2k},n_{2k}))}{\hookrightarrow}\lambda\big(2k+1;(p^{k}(p-1),i_{k+1})\big)
			\end{multline*} 
			where $0\leq i_{1}<p-1$, and $i_{j}=i_{1}+n_{2}+n_{4}+\cdots+n_{2(j-1)}$,  $m_{2(j-1)}=p^{j-2}(p-1)^{2}-((p-1)i_{2j-3}+t_{j})$, $n_{j-1} = (p-1)i_{2j-3}+t_{j}$, for $0\leq t_{j}\leq p-1$, for all $2\leq j\leq k+1$, and $k\geq 0$. \label{ak,bk}
			
			We make the convention that the path obtained by adding these blocks has no descent, since $m_{2(j-1)}<m_{2j}$ and $n_{2(j-1)}\leq n_{2j}$, for all $2\leq j \leq k+1$.
			\item 	For $k\geq 0$ and $r\geq k+1$, let $r-k=s$. We define a downward path starting at $\lambda\big(2r;(x_{(2s,k)},x_{(s,k)}+l)\big)$, where $0\leq l<p^{k}$, and ending at $\lambda\big(1;(p-1,i_{1})\big)$, where $0\leq i_{1}<p-1$ as follows:
			\begin{multline*}
				\lambda\big(2r;(x_{(2s,k)},x_{(s,k)}+l)\big)\overset{f_{1}}{\hookleftarrow}\lambda\big(2r-1;(x_{(2s-1,k)},x_{(s-1,k)}+(l+p^{k}t_{k}))\big)\overset{f_{2}}{\hookleftarrow} \\ \lambda\big(2r-2;(x_{(2s-2,k)},x_{(s-1,k)}+l_{1})\big)  \overset{f_{3}}{\hookleftarrow} \cdots    \lambda\big(2k+3;(x_{(3,k)},x_{(1,k)}+(l_{s-2}+p^{k}t_{r-3}))\big) \\ \overset{f_{2(r-k)-2}}{\hookleftarrow}\lambda\big(2k+2;(x_{(2,k)},x_{(1,k)}+l_{s-1})\big) 
				\overset{f_{2(r-k)-1}}{\hookleftarrow}\lambda\big(2k+1;(p^{k}(p-1),i_{k+1})\big){\hookleftarrow}\cdots \\ \overset{f_{2r-1}}{\hookleftarrow} \lambda\big(1;(p-1,i_{1})\big)
			\end{multline*}
			where $f_{d}=B(2r-d;(m_{2r-d},n_{2r-d}))$, $1\leq d\leq 2r-1$.
			
			For integers $t_{i}$ with $0\leq t_{i}\leq p-1$ for $0\leq i \leq k-1$, the integer $l$ can be written as $l=\sum\limits_{i=0}^{k-1}p^{i}t_{i}$. Then, for each $l$, there exist integers $t_{i}$ with $0\leq t_{i}\leq p-1$ for $k\leq i \leq r-2$, and $0\leq t_{r-1}\leq p-2$, such that the corresponding path ends at $\lambda\big(1;(p-1,t_{r-1})\big)$.
			
			Let $\alpha_{1}$ be the projection of $l$ (as in Remark \ref{Proj}), that is, $\alpha_{1}=\sum\limits_{i=1}^{k-1}p^{i-1}t_{i}$. Then: 
			\begin{itemize}
				\item $l_{j}=\alpha_{j}+p^{k-1}t_{k-1+j}$, $1\leq j \leq s-1$, and $\alpha_{j}=\sum\limits_{i=j}^{k+j-2}p^{i-j}t_{i}$ for $2\leq j \leq s-1$, where $\alpha_{j+1}$ is the projection of $l_{j}$ (as in Remark \ref{Proj}).
				\item For $k\geq 1$, $k+1\leq j\leq r-1$, and $0\leq t_{k+r-j-1} \leq p-1,$
				\begin{align*}
					m_{2j+1}&=p^{k}t_{k+r-j-1}, & n_{2j+1}&=p^{k}(p-1-t_{k+r-j-1}) \\
					m_{2j} & =p^{k}(p-1)-n_{2j-1}, & n_{2j} & = (p-1)(\alpha_{r-j}+p^{k-1}t_{k+r-j-1})+t_{r-j-1} 
				\end{align*}				
				and 
				\begin{align*}
					m_{2k+1}&=p^{k}t_{r-1}, & n_{2k+1}&=p^{k}(p-1-t_{r-1}), ~0\leq t_{r-1}\leq p-2.
				\end{align*}
				\item $i_{k+1}=l_{s-1}+p^{k}t_{r-1}=\sum\limits_{i=r-k-1}^{r-1}p^{i-(r-k-1)}t_{i}$ and for $1\leq j \leq k$, $$i_{j}= \sum\limits_{i=r-j}^{r-1}p^{i-(r-j)}t_{i} ~\text{is the projection of}~i_{j+1}= \sum\limits_{i=r-j-1}^{r-1}p^{i-(r-j-1)}t_{i}.$$
				\item For $1\leq j \leq k,$
				\begin{align*}
					m_{2j-1} & =0, & n_{2j-1} & = 0 \\
					m_{2j}&=p^{k-1}(p-1)^{2}-n_{2j}, & n_{2j}&=(p-1)i_{j}+t_{r-j-1} 
				\end{align*}				
				
			\end{itemize}
			In particular, when $k=0$, we have $l=0$ and $l_{j}=0$ for all $1\leq j \leq r-1$. Moreover,
			\begin{align*}
				m_{2j+1}&=t_{k+r-j-1}, & n_{2j+1}&=p-1-t_{k+r-j-1} \\
				m_{2j} & =p-1-t_{k+r-j-1}, & n_{2j} & =t_{r-j-1} \text{ and}\\
				m_{1}&=t_{r-1}, & n_{1}&=p-1-t_{r-1}, ~0\leq t_{r-1}\leq p-2.
			\end{align*}
			
			The downward path defined above is the same as the standard $p$-Young tableau of shape  $\lambda\big(2r;(x_{(2s,k)},x_{(s,k)}+l)\big)$ as defined in \cite{[TH1]}. 
			
			For $k\geq0$ and a given $l$, the path is determined by the integers $t_{i}$ for $0\leq i \leq r-1$. For simplicity, we rewrite the above path as $$\left(\lambda\big(2r;(x_{(2s,k)},x_{(s,k)}+l)\big),m_{2r-1},m_{2r-2},\ldots,m_{1},\lambda\big(1;(p-1,t_{r-1})\big)\right)$$ where $i_{1}=t_{r-1}$. 
			\label{blocklabel}	
			\item Let $P^{2r}_{l,m_{2r-1},\ldots,m_{1},t_{r-1}}$ denote the path $$\left(\lambda\big(2r;(x_{(2s,k)},x_{(s,k)}+l)\big),m_{2r-1},\ldots,m_{1},\lambda\big(1;(p-1,t_{r-1})\big)\right)$$ as defined above.
			Let $P^{2r}_{l,t_{r-1}}$ denote the set of all paths starting at the vertex $\lambda\big(2r;(x_{(2s,k)},x_{(s,k)}+l)\big)$ and ending at the vertex $\lambda\big(1;(p-1,t_{r-1})\big)$, $0\leq t_{r-1}<p-1$, which are obtained by removing the blocks successively. Let $P^{2r}_{l}=\bigcup\limits_{t_{r-1}=0}^{p-2}P^{2r}_{l,t_{r-1}}$. \label{pathset}
				\end{enumerate}		
	\end{remark}
	\section{Inversions and descents of paths }
	In this section, we define the notions of inversion and descent of a path in the $p$-Bratteli diagram using the blocks added along the path. Inversions (or descents) are determined by comparing the horizontal nodes or vertical nodes of the blocks of the same size, namely $B(i;(m_{i},n_{i}))$ and $B(j;(m_{j},n_{j}))~(\text{or}~B(i+1;(m_{i+1},n_{i+1})))$, for $i<j$. We do not compare blocks of different sizes except in the following case. 

For the blocks $B(2k;(m_{2k},n_{2k}))$ and $B(2k+1;(m_{2k+1},n_{2k+1}))$ of sizes $p^{k}(p-1)^{2}$ and $p^{k}(p-2)$, we introduce a specific convention, since the sizes of the respective blocks are different,
\begin{enumerate}[label=\textbf{(c\arabic*)}]
	\item \label{c1} Consider the path $\lambda\big(2k+2;(x_{(2,k)},x_{(1,k)}+l)\big)\overset{B(2k+1;(p^{k}t,p^{k}(p-2-t)))}{\hookleftarrow} \lambda\big(2k+1;(p^{k}(p-1),l+p^{k}t)\big),~0\leq t \leq p-2$, we say that there is a descent (or inversion) at $2k$ of the path if it is obtained by removing the block $B(2k+1;(p^{k}t,p^{k}(p-2-t)))$, where $0\leq t <\frac{p-1}{2};$ otherwise, there is no descent (or inversion). When $k=0$, the zeroth block is the hook partition $\lambda(1;(p-1,t_{r-1}))$, where $0\leq t_{r-1}\leq p-2.$ 
	\item \label{c2} Since the blocks $B(2k+1;m_{2k+1},n_{2k+1}))$ and $B(2k+2;(m_{2k+2},n_{2k+2}))$ have different sizes, $p^{k}(p-2)$ and $p^{k}(p-1)$, respectively, they are not comparable; therefore, there is no descent (or inversion) at block $2k+1$.
\end{enumerate}
\begin{definition}\label{block1}
	For $k\geq 0$ and $2k+2\leq i<j$, we write  $B(i;(m_{i},n_{i}))>B(j;(m_{j},n_{j}))$ if $m_{i}>m_{j}$ and $n_{i}<n_{j}$, where $m_{i}+n_{i}= m_{j}+n_{j}$.
\end{definition}
\subsection{Inversions and sign balances of paths}
\begin{definition}\label{inversion}
	Let $k\geq 0$, and let $P$ be the path $$\left(\lambda\big(2r;(x_{(2s,k)},x_{(s,k)}+l)\big),m_{2r-1},m_{2r-2},\ldots,m_{1},\lambda\big(1;(p-1,t_{r-1})\big)\right)$$ as in Remark \ref{syt def}(\ref{blocklabel}). Also, let $B(i;(m_{i},n_{i}))$, $B(j;(m_{j},n_{j}))$ be the blocks added along the path $P$, with $i<j$.  

The \textit{inversion of a path $P$} is defined as follows:
	\begin{displaymath}
		Inv(P)=
		\begin{cases}
			\{(i,j):~ i<j ~\text{but $B(i;(m_{i},n_{i}))>B(j;(m_{j},n_{j}))$}\}, & \text{for $i\geq 2k+2$;} \\  \text{inversion as in \ref{c1} and \ref{c2}}, & \text{for}~i=2k,~2k+1
		\end{cases}
	\end{displaymath}
	where $m_{i}+n_{i}= m_{j}+n_{j}$ for $i\geq 2k+2$ and $inv(P)=|Inv(P)|$.
\end{definition}

\begin{definition}
	We define the \textit{sign of a path $P$} by $sgn(P) = (-1)^{inv(P)}$.
\end{definition}
 This sign balance is a structural property of paths in the $p$-Bratteli diagram.
\begin{definition}
		Given a partition $\lambda$, define its sign balance by $$I_{\lambda}=\sum_{P\in P^{2r}_{l}}sgn(P),$$ where $\lambda=\lambda\big(2r;(x_{(2s,k)},x_{(s,k)}+l)\big)$ and $P^{2r}_{l}$ as defined in Remark \ref{syt def}(\ref{pathset}).
\end{definition}
The following theorem motivates our convention \ref{c1}.
\begin{theorem}
	Let $k\geq0$, $r\geq 1$ with $s=r-k.$ The sign balance for every vertex of shape $\lambda\big(2r;(x_{(2s,k)},x_{(s,k)}+l)\big)\in V^{2r}_{k}$ is zero.
\end{theorem}
\begin{proof}
	Let $P$ be a path $$\left(\lambda\big(2r;(x_{(2s,k)},x_{(s,k)}+l)\big),m_{2r-1},m_{2r-2},\ldots,m_{1},\lambda\big(1;(p-1,t_{r-1})\big)\right)$$ as defined in Remark \ref{syt def}(\ref{blocklabel}), and let $j$ denote total number of inversions of the path $P$ from block $2k+2$ to block $2r-2$. 
		
	By convention \ref{c1}, the total number of inversions of the path from block $2k$ is
	\begin{displaymath}
		\begin{cases}
			j+1, & \text{if $m_{2k+1}=p^{k}t$, $0\leq t<\frac{p-1}{2}$}; \\ 
			j, & \text{if $m_{2k+1}=p^{k}t$, $\frac{p-1}{2}\leq t\leq p-2$}.
		\end{cases}
	\end{displaymath}
	Consequently, when $j$ is odd,
	\begin{displaymath}
		(-1)^{inv(P)}=
		\begin{cases}
			1, & \text{if $m_{2k+1}=p^{k}t$, $0\leq t<\frac{p-1}{2}$}; \\ 
			-1, & \text{if $m_{2k+1}=p^{k}t$, $\frac{p-1}{2}\leq t\leq p-2$}.
		\end{cases}
	\end{displaymath}
	and when $j$ is even,
	\begin{displaymath}
		(-1)^{inv(P)}=
		\begin{cases}
			-1, & \text{if $m_{2k+1}=p^{k}t$, $0\leq t<\frac{p-1}{2}$}; \\ 
			1, & \text{if $m_{2k+1}=p^{k}t$, $\frac{p-1}{2}\leq t\leq p-2$}.
		\end{cases}
	\end{displaymath}
	Therefore, 	$$\sum_{P\in P^{2r}_{l}}sgn(P)=\sum_{P\in P^{2r}_{l}}(-1)^{inv(P)}=0.$$		
\end{proof}

	\subsection{Descents of paths}
	\begin{definition}\label{block}
		Let $k\geq 0$, and let $P$ be the path $$\left(\lambda\big(2r;(x_{(2s,k)},x_{(s,k)}+l)\big),m_{2r-1},m_{2r-2},\ldots,m_{1},\lambda\big(1;(p-1,t_{r-1})\big)\right)$$ as in Remark \ref{syt def}(\ref{blocklabel}). Also, let $B(i;(m_{i},n_{i}))$, $B(i+1;(m_{i+1},n_{i+1}))$ be the blocks added along the path $P$. We define the \textit{descent set of $P$} as follows:
		\begin{displaymath}
			Des(P)=
			\begin{cases}
				\{i:~ B(i;(m_{i},n_{i}))>B(i+1;(m_{i+1},n_{i+1}))\}, & \text{for $i\geq 2k+2$;} \\  \text{descent as in \ref{c1}and \ref{c2}}, & \text{for $i=2k$, $2k+1$}
			\end{cases}
		\end{displaymath}
		where $m_{i}+n_{i}= m_{i+1}+n_{i+1}$ for $i\geq 2k+2$ and $des(P)=|Des(P)|$. If $i\in Des(P)$, then we say that there is a descent at the $i$-th block of path $P$. Equivalently, $Des(P)$ contains descents up to the $(2r-2)$-th block of the path $P$.
	\end{definition}
	\begin{lemma}\label{2p-3}
		Let $k\geq 0$, and let $$P=\left(\lambda\big(2k+2;(x_{(2,k)},x_{(1,k)}+l)\big),m_{2k+1},m_{2k},\ldots,m_{1},\lambda\big(1;(p-1,t_{r-1})\big)\right),$$ and 
		$$P^{'}=\left(\lambda\big(2k+3;(x_{(3,k)},x_{(1,k)}+l^{'})\big),m_{2k+2},m_{2k+1},\ldots,m_{1},\lambda\big(1;(p-1,t_{r-1})\big)\right)$$ be paths in the $p$-Bratteli diagram, where $0\leq l <p^{k}$ and $0\leq l^{'} <p^{k+1}$. Then 
		\begin{displaymath}
			des(P) = \begin{cases}
				1, & \text{when $m_{2k+1}=p^{k}t,~0\leq t < \frac{p-1}{2}$;} \\ 0, & \text{when $m_{2k+1}=p^{k}t,~\frac{p-1}{2}\leq t \leq p-2$.}
			\end{cases}
		\end{displaymath}
		Moreover, $des(P^{'})=des(P)$ for the same $m_{2k+1}$.	Furthermore, 
		the total number of descents of all paths ending at the vertex $\lambda\big(2k+2;(x_{(2,k)},x_{(1,k)}+l)\big)$ \big(and at $\lambda\big(2k+3;(x_{(3,k)},x_{(1,k)}+l^{'})\big)$\big) is $\frac{p-1}{2}$.
	\end{lemma}
	\begin{proof}
		 The paths ending at the vertex  $\lambda\big(2k+2;(x_{(2,k)},x_{(1,k)}+l)\big)$ are illustrated in Figure \ref{figlemma}
		 
		\begin{figure}[!ht]
			\begin{center}
				\includegraphics[width=0.6\linewidth]{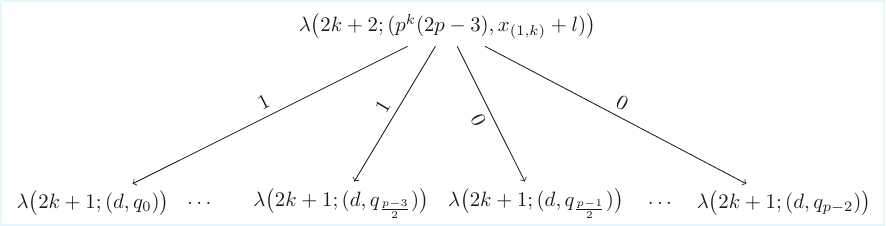}
			\end{center}
			\caption{}
			\label{figlemma}
		\end{figure}
		
		where $d=p^{k}(p-1)$ and $q_{t}= x_{(1,k)}+l-p^{k}(p-2-t)$, $0\leq t\leq p-2$.
		
		The descent at block $2k$ between the vertices follows convention \ref{c1}. Therefore, $des(P)$ is either 1 or 0.
		By Definition \ref{block}, the $des(P^{'})$ is also the same as $des(P)$, hence $des(P^{'})$ is either 1 or 0.  
		
		The total number of descents of all paths ending at the given vertex is $\frac{p-1}{2}$. Hence the proof follows.
	\end{proof}
	\begin{lemma}\label{thmlem}
		\begin{enumerate}
			\item  The blocks added to the partition $\lambda\big(2r-2;(x_{(2s,k)},x_{(s,k)}+l)\big)$ to obtain the partition $\lambda\big(2r-1;(x_{(2s+1,k)},x_{(s,k)}+l^{'})\big)$ satisfy the following inequality 	
			$$B(2r-2;(p^{k}(p-1)-l_{1},l_{1}))>B(2r-2;(p^{k}(p-1)-l_{2},l_{2})),~\forall~0\leq l_{1}<l_{2}\leq p^{k}(p-1)$$
			where $0\leq l < p^{k}$, $0\leq l^{'}<p^{k+1}$, for all integers $s\geq 1,~k\ge 0$, and where $x_{(a,k)}$ is as in Notation \ref{not}(\ref{x_{a,k}}).
			\item The blocks added to the partition $\lambda\big(2r-1;(x_{(2s-1,k)},x_{(s-1,k)}+l^{'})\big)$ to obtain the partition $\lambda\big(2r;(x_{(2s,k)},x_{(s,k)}+l)\big)$ satisfy the following inequality
			$$B(2r-1;(p^{k}t_{1},p^{k}(p-1-t_{1})))<B(2r-1;(p^{k}t_{2},p^{k}(p-1-t_{2}))),~\forall~0\leq t_{1}<t_{2}\leq p-1.$$
			where $0\leq l < p^{k}$, $0\leq l^{'} < p^{k+1}$, for all $s\geq 2$ and $x_{(a,k)}$ is as in Notation \ref{not}(\ref{x_{a,k}}).
		\end{enumerate}
	\end{lemma}	
	\begin{proof}
\begin{enumerate}
  \item Since $l_{1}<l_{2}$, $p^{k}(p-1)-l_{1}>p^{k}(p-1)-l_{2}$.   
  Thus the first coordinate of the first block is larger than the first coordinate of the second block.
By definition \ref{block}, blocks are ordered by their first coordinate. 
Hence the first block is greater than the second block. 
\item The proof is analogous to that of part (1).
\end{enumerate}
	\end{proof}
	The following lemma and theorem will be useful in determining the descents for every path in the $p$-Bratteli diagram. 

We consider only the vertices of the floors $2r,~2r-1$, and $2r-2$ for $r\ge 2$.
	\begin{lemma}\label{rules jt}
		For $k\geq 1$ and $0\leq t\leq p-1$, let $P^{2r}_{l}$ denote the paths $$\left(\lambda\big(2r;(x_{(2s,k)},x_{(s,k)}+l)\big),m_{2r-1},\ldots,m_{1},\lambda\big(1;(p-1,t_{r-1})\big)\right),$$ as defined in Remark~\ref{syt def}(\ref{blocklabel}), for $1\le j\le k$, where
		$m_{2d+1}=p^{k}t$, and $m_{2d}=p^{k}(p-1-t)$, for $k+1\leq d\leq r-1$.
		Then the following conditions hold for the path $P^{2r}_{l}$
		\begin{itemize}
			\item If $0\leq t <\frac{p-1}{2}$, then there is a descent at the $(2r-2)$-nd block and no descent at the $(2r-3)$-rd block of the paths $P^{2r}_{l}$.
			\item If $ t =\frac{p-1}{2}$, then there is no descent at the blocks $2r-2$ and $2r-3$ for every $s$, of the paths $P^{2r}_{l}$.
			\item If $\frac{p-1}{2}<t\leq p-1$, then there is no descent at the $(2r-2)$-nd block and there is a descent at the $(2r-3)$-rd block of the paths $P^{2r}_{l}$. 
		\end{itemize} 
		Here $l=j^{k}_{t}$.
	\end{lemma}	
	\begin{proof}
		\begin{itemize}
			\item When $0\leq t <\frac{p-1}{2}$, we have $m_{2r-2}>m_{2r-1} ~\text{and} ~m_{2r-3}<m_{2r-2}$. 
			Hence, the paths $P^{2r}_{l}$ have a descent at the $(2r-2)$-nd block and no descent at the $(2r-3)$-rd block. 
			\item When $t=\frac{p-1}{2}$, we have   $m_{2r-2}=m_{2r-1} ~\text{and} ~m_{2r-3}=m_{2r-2}$.			
			Hence, the paths $P^{2r}_{l}$ have no descent at the blocks $2r-2$ and $2r-3$.
			\item When $\frac{p-1}{2}<t\leq p-1$, we have $
				m_{2r-2}<m_{2r-1} ~\text{and} ~m_{2r-3}>m_{2r-2}$.
			Hence, the paths $P^{2r}_{l}$ have no descent at the $(2r-2)$-nd block and have a descent $(2r-3)$-rd block.
		\end{itemize}
	\end{proof} 
	The above paths are \textbf{special paths} of the $p$-Bratteli diagram and help determine the descents of any paths.
	
	For $k=0$ and $r\geq 2$, let $P^{2r}_{0,t^{'},t^{''}}$ denote any path in the $p$-Bratteli diagram starts at $\lambda\big(2r;(x_{(2r,0)},x_{(r,0)})\big)$ and ends at $\lambda\big(1;(p-1,t_{r-1})\big)$, for some $0\leq t_{r-1}\leq p-2$, obtained by removing the blocks as described in Remark \ref{syt def}(\ref{blocklabel}). 

We write this path as  $$(\lambda\big(2r;(x_{(2r,0)},x_{(r,0)})\big),t^{'},p-t^{'}-1,t^{''},m_{2r-4},\ldots,m_{1},\lambda\big(1;(p-1,t_{r-1})\big) ).$$ That is,
	$m_{2r-3}=t^{''}$, $m_{2r-2}=p-1-t^{'}$ and $m_{2r-1}=t^{'}$, where $0\leq t^{'},t^{''}\leq p-1$. 
	
	For $k\geq 1$ and $r\geq3$ with $s=r-k\geq 2$, let $P^{2r}_{l,t^{'},t^{''}}$ be any path in the $p$-Bratteli diagram that starts at $\lambda\big(2r;(x_{(2s,k)},x_{(s,k)}+l)\big)$ and ends at $\lambda\big(1;(p-1,t_{r-1})\big)$, for some $0\leq t_{r-1}\leq p-2$, obtained by removing the blocks as described in Remark \ref{syt def}(\ref{blocklabel}). 

Such a path is denoted by
	\begin{multline*}
		\Big(\lambda\big(2r;(x_{(2s,k)},x_{(s,k)}+l)\big),p^{k}t^{'},p^{k}(p-1)- (\alpha+p^{k-1}t^{'})(p-1)+\beta,p^{k}t^{''},m_{2r-4},\ldots,m_{1},  \\  \lambda\big(1;(p-1,t_{r-1})\big)\Big)
	\end{multline*}
	 that is
	\begin{itemize}
		\item $m_{2r-1}=p^{k}t^{'}$
		\item $m_{2r-2}=p^{k}(p-1)- (\alpha+p^{k-1}t^{'})(p-1)+\beta$
		\item $m_{2r-3}=p^{k}t^{''}$.
	\end{itemize}
	Here $l=\alpha p+\beta$, $0\leq \alpha<p^{k-1}$, and $0\leq \beta,t^{'},t^{''}\leq p-1$. The notation $x_{(a,k)}$ is as defined in Notation \ref{not}(\ref{x_{a,k}}). 
	
	A few paths across three floors starting at the vertex $\lambda\big(2r;(x_{(2s,k)},x_{(s,k)}+l)\big)$ on the $2r$-th floor, are illustrated in Figure \ref{thmfigure}
	\begin{figure}[!ht]
		\centering
		\includegraphics[width=12cm,height=9cm]{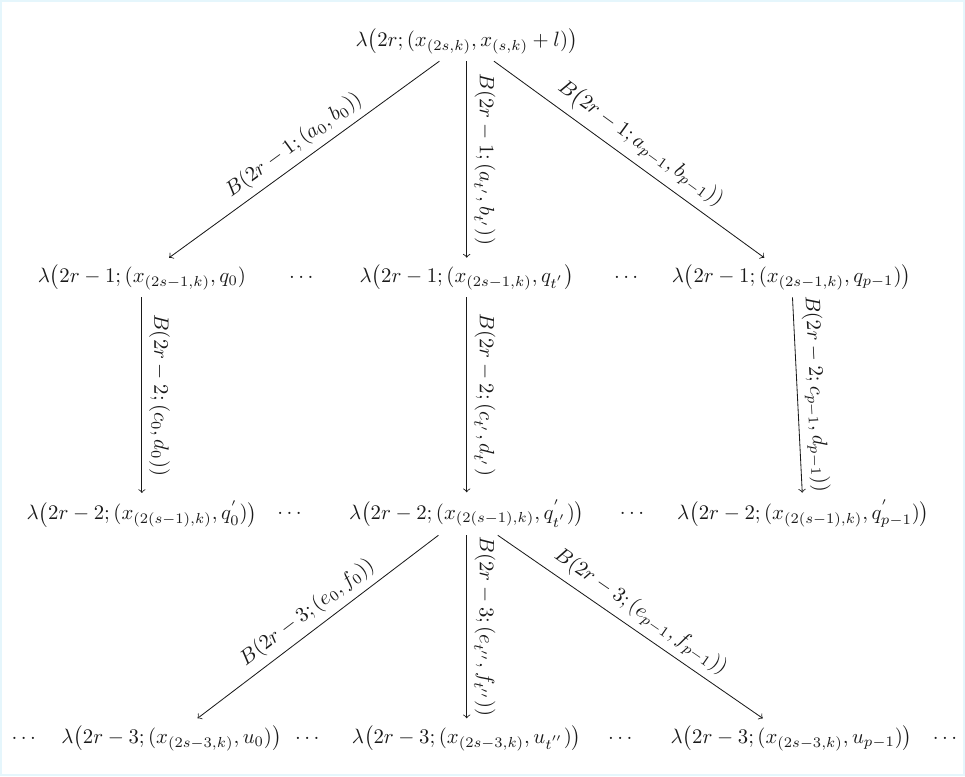}
		\caption{}
		\label{thmfigure}
	\end{figure}  
	
	For $0\leq t^{'}\leq p-1$,
	\begin{align*}
		q_{t^{'}}&=x_{(s-1,k)}+(l+p^{k}t^{'}), \\ q^{'}_{t^{'}}&=x_{(s-1,k)}+(\alpha+p^{k-1}t^{'}), \\ a_{t^{'}}&=p^{k}t^{'},~ b_{t^{'}}=p^{k}(p-1-t^{'}),\\ c_{t^{'}}&=p^{k}(p-1)-(\alpha+p^{k-1}t^{'})(p-1)+\beta,\\ d_{t^{'}}&=(\alpha+p^{k-1}t^{'})(p-1)+\beta,
	\end{align*}
	and for $0\leq t^{''}\leq p-1$,
	\begin{align*}
		u_{t^{''}}&=x_{(s-2,k)}+(\alpha+p^{k-1}t^{'}+p^{k}t^{''}) \\   e_{t^{''}}&=p^{k}t^{''},~f_{t^{''}}=p^{k}(p-1-t^{''}).
			\end{align*}
	   	\begin{definition}\label{comp}
		For $k\geq 1$, the vertices $\lambda\big(2r-2;(x_{(2(s-1),k)},x_{(s-1,k)}+(\alpha+p^{k-1}t^{'}))\big)$, $0\leq t^{'}\leq p-1$ are called the \textit{components} of the vertex $\lambda\big(2r;(x_{(2s,k)},x_{(s,k)}+l)\big)$, where the position of the components are precisely $\alpha+p^{k-1}t^{'}$.
	\end{definition}
	\begin{theorem}\label{rules l}
		For $r\geq2$, let $P^{2r}_{l,t^{'},t^{''}}$ denote the paths defined above. Then the following hold:
		\begin{enumerate}
			\item If $l<\frac{p^{k}-1}{2}$ and $k\neq0$, the paths $P^{2r}_{l,t^{'},t^{''}}$ have a descent at block $2r-2$ only for  $0\leq t^{'}\leq\frac{p-1}{2}$ and $0\leq t^{''}\leq p-1$.    
			\item If $l\geq \frac{p^{k}-1}{2}$ and $k\neq0$, the paths $P^{2r}_{l,t^{'},t^{''}}$ have a descent at block $2r-2$ only for $0\leq t^{'}\leq \frac{p-3}{2}$ and $0\leq t^{''}\leq p-1$. 
			\item If $j^{k}_{t-1}< l\leq j^{k}_{t},~k\neq0$, and $1\leq t\leq p-1$, the paths $P^{2r}_{l,t^{'},t^{''}}$ have a descent at block $2r-3$ only for $0\leq t^{'}\leq t-1$ and $p-1-t^{'}\leq t^{''}\leq p-1$. 
			\item If $j^{k}_{t-1}< l\leq j^{k}_{t},~k\neq0$, and $1\leq t\leq p-1$, the paths $P^{2r}_{l,t^{'},t^{''}}$ have a descent at block $2r-3$ only for $t\leq t^{'}\leq p-1$ and $p-t^{'}\leq t^{''}\leq p-1$.
			\item If $ l=j^{k}_{0}$ and $k\neq0$, the paths $P^{2r}_{l,t^{'},t^{''}}$ have a descent at block $2r-3$ only for $1\leq t^{'}\leq p-1$ and $p-t^{'}\leq t^{''}\leq p-1$. 

For $t^{'}=0$, the paths have no descent at block $2r-3$ for every $0\leq t^{''}\leq p-1$. 	
			\item If $k=0$ and $0\leq t^{'} \leq \frac{p-3}{2}$, the paths $P^{2r}_{0,t^{'},t^{''}}$ have a descent at block $2r-2$.
			\item If $k=0$ and $1\leq t^{'}\leq p-1$, the paths $P^{2r}_{0,t^{'},t^{''}}$ have a descent at block $2r-3$ only when $p-t^{'}\leq t^{''}\leq p-1$. 

When $t^{'}=0$, the paths have no descent at block $2r-3$.\label{bb}
			\end{enumerate} 
	\end{theorem}	
	\begin{proof}
		\begin{enumerate}
			\item  When $l<\frac{p^{k}-1}{2}$, we have $l+p^{k}t^{'}<\frac{p^{k+1}-1}{2}$ for $0\leq t^{'}\leq \frac{p-1}{2};$ that is,
			\begin{multline*}
				\lambda\big(2r-1;(x_{(2s-1,k)},x_{(s-1,k)}+(l+p^{k}t^{'}))\big)\\ \trianglerighteq\lambda\Bigg(2r-1;\Big(x_{(2s-1,k)},x_{(s-1,k)}+\Big(\frac{p^{k+1}-1}{2}\Big)\Big)\Bigg)
			\end{multline*}
			By Lemma \ref{thmlem}, we have $$m_{2r-2}=p^{k}(p-1)-(\alpha+p^{k-1}t^{'})(p-1)+\beta>\frac{p^{k}(p-1)}{2}\geq p^{k}t^{'}=m_{2r-1},$$ for $0\leq t^{'}\leq \frac{p-1}{2}$. 
			
			Therefore, the paths $P^{2r}_{l,t^{'},t^{''}}$ have a descent at the block $2r-2$ only for $0\leq t^{'}\leq \frac{p-1}{2}$.
			\item  When $l\geq\frac{p^{k}-1}{2}$, we have $l+p^{k}t^{'}<\frac{p^{k+1}-1}{2}$ for $0\leq t^{'}< \frac{p-1}{2}$ (i.e) $$\lambda\big(2r-1;(x_{(2s-1,k)},x_{(s-1,k)}+(l+p^{k}t^{'}))\big)\trianglerighteq\lambda\Bigg(2r-1;\Big(x_{(2s,k)},x_{(s-1,k)}+\Big(\frac{p^{k+1}-1}{2}\Big)\Big)\Bigg)$$
			By Lemma \ref{thmlem}, we have $$m_{2r-2}=p^{k}(p-1)-(\alpha+p^{k-1}t^{'})(p-1)+\beta>\frac{p^{k}(p-1)}{2}\geq p^{k}t^{'}=m_{2r-1},$$ when $0\leq t^{'}\leq \frac{p-3}{2}$. 
			
			Therefore, the paths $P^{2r}_{l,t^{'},t^{''}}$ have a descent at the block $2r-2$ only for $0\leq t^{'}\leq \frac{p-3}{2}$.
			
			For $j^{k}_{t-1}< l\leq j^{k}_{t},~1\leq t \leq p-1$ and $j^{k+1}_{t}$ in the floor $2r-1$, the following relations hold:  
			
		 When $0\leq t^{'}\leq t-1$, 
			\begin{align}\label{l1}
				j^{k}_{t}+p^{k}t^{'}&< l+p^{k}t^{'}\leq j^{k}_{t+1}+p^{k}t^{'} \notag \\ j^{k+1}_{t^{'}}&<l+p^{k}t^{'}<j^{k+1}_{t^{'}+1}
			\end{align}
		 When $t\leq t^{'}\leq p-1$, 
			\begin{align}\label{l4}
				j^{k}_{t}+p^{k}t^{'}&< l+p^{k}t^{'}\leq j^{k}_{t+1}+p^{k}t^{'} \notag \\j^{k+1}_{t^{'}-1}&<l+p^{k}t^{'}\leq j^{k+1}_{t^{'}}
			\end{align}
			Note that $$j^{k+1}_{t-1}<l+p^{k}(t-1),~l+p^{k}t\leq j^{k+1}_{t}.$$ Since $l\leq j^{k}_{t}$, it follows that $l+p^{k}t\leq j^{k}_{t}+p^{k}t=j^{k+1}_{t}$. 
			
			\item  When $j^{k}_{t-1}<l\leq j^{k}_{t}$ and $0\leq t^{'}\leq t-1$, equation \eqref{l1} and Lemma \ref{thmlem} imply that
			\begin{equation}\label{l7}
				\begin{aligned}
					B(2r-2;(p^{k}(p-1-t^{'}),p^{k}t^{'}))&>B(2r-2;(c_{t^{'}},d_{t^{'}})) \\ &>B(2r-2;(p^{k}(p-2-t^{'}),p^{k}(t^{'}+1)))
				\end{aligned}
			\end{equation}
			\begin{equation*}
				\begin{aligned}
					B(2r-3;(p^{k}t^{''},p^{k}(p-1-t^{''}))) &\geq B(2r-2;(p^{k}(p-1-t^{'}),p^{k}t^{'}))\\&>B(2r-2;(c_{t^{'}},d_{t^{'}}))   
				\end{aligned}
			\end{equation*}
			The above inequality holds when $t^{''}=p-1,~p-2,~\ldots,p-1-t^{'}$. 
			
			Therefore, for such values of $t^{'}$ and $t^{''}$, the paths $P^{2r}_{l,t^{'},t^{''}}$ have a descent at the block  $2r-3$. 
			
			\item    When $j^{k}_{t-1}<l\leq j^{k}_{t}$ and $t\leq t^{'}\leq p-1$, equation \eqref{l4} and Lemma \ref{thmlem} imply that 
			\begin{equation}\label{l9}
				\begin{aligned}
						B(2r-2;(p^{k}(p-t^{'}),p^{k}(t^{'}-1)))&>B(2r-2;(c_{t^{'}},d_{t^{'}}))\\&>B(2r-2;(p^{k}(p-1-t^{'}),p^{k}t^{'}))
				\end{aligned}
			\end{equation}
			\begin{equation*}
					\begin{aligned}
					B(2r-3;(p^{k}t^{''},p^{k}(p-1-t^{''}))) &\geq B(2r-2;(p^{k}(p-t^{'}),p^{k}(t^{'}-1)))\\&>B(2r-2;(c_{t^{'}},d_{t^{'}}))		
				\end{aligned}			
			\end{equation*}
			The above inequality holds when $t^{''}=p-1,~p-2,~\ldots,p-t^{'}$. 
			
			Therefore, for such values of $t^{'}$ and $t^{''}$, the paths $P^{2r}_{l,t^{'},t^{''}}$ have a descent at the block $2r-3$. 
			\item When $l=j^{k}_{0}$, and $t^{'}=0$, we have $l+p^{k}t^{'}=j^{k+1}_{0}$ and the corresponding $(2r-2)$-th block added to this vertex is $B(2r-2;(p^{k}(p-1),0))$. No block of the form satisfies $B(2r-3;(p^{k}t^{''},p^{k}(p-1-t^{''})))$ satisfies $$B(2r-3;(p^{k}t^{''},p^{k}(p-1-t^{''})))>B(2r-2;(p^{k}(p-1),0)),~\text{for all}~0\leq t^{''}\leq p-1.$$ Hence, the paths $P^{2r}_{l,t^{'},t^{''}}$ have no descent at the block $2r-3$.
			
			When $l=j^{k}_{0}$, and $1\leq t^{'}\leq p-1$, equations (\ref{l4}), (\ref{l9}) with the substitutions $\alpha=0$, $\beta=0$, yields		
			\begin{equation*}
				\begin{aligned}
					B(2r-3;(p^{k}t^{''},p^{k}(p-1-t^{''}))) &\geq
					B(2r-2;(p^{k}(p-t^{'}),p^{k}(t^{'}-1)))\\ &\hspace*{-0.8cm}>B(2r-2;(p^{k}(p-1)-p^{k-1}t^{'}(p-1),p^{k-1}t^{'}(p-1)))		
				\end{aligned}
			\end{equation*}			
			The above inequality holds when $t^{''}=p-1,~p-2,~\ldots,p-t^{'}$. 
			
			Therefore, for such values of $t^{'}$ and $t^{''}$, the paths $P^{2r}_{l,t^{'},t^{''}}$ have a descent at the block $2r-3$. 
			\item When $0\leq t^{'} \leq \frac{p-3}{2}$, we have 
			$$m_{2r-2}=p-1-t^{'}>t^{'}=m_{2r-1}.$$ Hence, the paths $P^{2r}_{0,t^{'},t^{''}}$ have a descent at the block $2r-2$.			\item When $1\leq t^{'}\leq p-1$, we have 
			$$m_{2r-3}=t^{''}>p-1-t^{'}=m_{2r-2}$$
			only for $p-t^{'}\leq  t^{''}\leq p-1$. Therefore, for such values of $t^{'}$ and $t^{''}$, the paths $P^{2r}_{0,t^{'},t^{''}}$ have a descent at the block $2r-3$.
			
			When $t^{'}=0$, we have $$m_{2r-2}=p-1\geq t^{''}=m_{2r-3},~0\leq t^{''}\leq p-1.$$ Hence, the paths $P^{2r}_{0,t^{'},t^{''}}$ have no descent at the block $2r-3$.
		\end{enumerate}		
	\end{proof}
	
	\begin{corollary}\label{cor rules l}
		Let the paths be as in Theorem~\ref{rules l}. Then:
		\begin{enumerate}
			\item The total number of descents at the $(2r-2)$-th block of all paths ending at the vertex $\lambda\big(2r;(x_{(2s,k)},x_{(s,k)}+l)\big)$ is
			\begin{align*}
				&p^{s-1}(p-1)\left(\frac{p+1}{2}\right)\quad \text{when}\quad 0\leq l<\frac{p^{k}-1}{2},~k\geq 1, \\  &p^{s-1}(p-1)\left(\frac{p-1}{2}\right)\quad \text{when} \quad\frac{p^{k}-1}{2}\leq l <p^{k},~k\geq 1, \\
				&p^{s-1}(p-1)\left(\frac{p-1}{2}\right)\quad \text{when} \quad k=0.
			\end{align*}\label{2s-1}
			
			\item The total number of descents at the $(2r-3)$-th block of all paths ending at the vertex $\lambda\big(2r;(x_{(2s,k)},x_{(s,k)}+l)\big)$ is $$p^{s-2}(p-1)\left(\frac{p(p-1)}{2}+t\right)\quad \text{when} \quad j^{k}_{t-1}< l\leq j^{k}_{t},~ 1\leq t\leq p-1,~k\geq 1,$$ and $$p^{s-2}(p-1)\left(\frac{p(p-1)}{2}\right)\quad \text{when}\quad l =j^{k}_{0},~ k\geq 0.$$\label{2s-2}
			
		\end{enumerate}
	\end{corollary}
	\begin{proof}
		The descent at the blocks $2r-2$ and $2r-3$ depends only on the blocks		$B(2r-3;(m_{2r-3},n_{2r-3}))$, $B(2r-2;(m_{2r-2},n_{2r-2}))$, $B(2r-1;(m_{2r-1},n_{2r-1}))$ of each path $P^{2r}_{l,t^{'},t^{''}}$. 
		
		There are $p^{s-1}(p-1)$ paths up to the $(2r-2)$-th block, that connect to the vertex $\lambda\big(2r;(x_{(2s,k)},x_{(s,k)}+l)\big)$. Similarly, there are $p^{s-2}(p-1)$ paths up to the $(2r-3)$-th block, that connect to the same vertex.
		Thus, the descent at the blocks $2r-2$ and $2r-3$ is multiplied across all  these distinct paths from the vertex $\lambda\big(2r;(x_{(2s,k)},x_{(s,k)}+l)\big)$ to the vertex $\lambda\big(1;(p-1,t_{r-1})\big)$, $0\leq t_{r-1} <p-1$.
		
		By Theorem \ref{rules l}, for $k\geq 1$, the number of descents at the $(2r-2)$-th block equals $\left(\frac{p+1}{2}\right)$ when $0\leq l<\frac{p^{k}-1}{2}$ and  $\left(\frac{p-1}{2}\right)$ when $\frac{p^{k}-1}{2}\leq l <p^{k}$. For $k=0$, the number of descents at $(2r-2)$-th block always equals $\left(\frac{p-1}{2}\right)$. Hence, this proves (\ref{2s-1}).
		
		Similarly, for $k\geq0$, by Theorem \ref{rules l}, the number of descents at the $(2r-3)$-th block equals $t^{'}+1$ when $t^{'}\leq t-1$, and $t^{'}$ when $t^{'}\geq t$, for $0\leq t^{'}\leq p-1$. Consequently, the total number of descents at the $(2r-3)$-th block equals $\left(\frac{p(p-1)}{2}+t\right)$, $0\leq t\leq p-1$, which proves (\ref{2s-2}).	
	\end{proof}
	\section{The \texorpdfstring{$p^{(k)}$}{p(k)}-Fibonacci numbers}
	In this section, we define the $p^{(k)}$-Fibonacci number for each vertex in the set $V^{2r}_{k}$ where $0\leq k\leq r-1$ and $r\geq1$. We compute the  $p^{(k)}$-Fibonacci numbers using mathematical induction. Furthermore, we identify the subclasses of the vertex set $V^{2r}_{k}$ in which the $p^{(k)}$-Fibonacci numbers remain the same. 
	
	\begin{definition}\label{Fibo}
		For $0\leq k\leq r-1$ and $r\geq 2$, let $s=r-k$. We define the \textit{$p^{(k)}$-Fibonacci number} is the total number of descents up to the $(2r-2)$-th block of all paths in $P^{2r}_{l}$ (as defined in Remark \ref{syt def}(\ref{pathset})). We denote this number by $\mathcal{M}\Big(\lambda\big(2r;(x_{(2s,k)},x_{(s,k)}+l)\big)\Big);$ that is, $$\mathcal{M}\Big(\lambda\big(2r;(x_{(2s,k)},x_{(s,k)}+l)\big)\Big)=\sum_{P^{2r}_{l,m_{2r-1},\ldots,m_{1},t_{r-1}}\in P^{2r}_{l}} des(P^{2r}_{l,m_{2r-1},\ldots,m_{1},t_{r-1}}).$$
		We can rewrite this quantity as the sum of all descents of the components of the vertex $\lambda\big(2r;(x_{(2s,k)},x_{(s,k)}+l)\big)$ (as in Definition \ref{comp}) up to the $(2r-4)$-th block of all paths in $P^{2r-2}_{\alpha+p^{k-1}t^{'}},~0\leq t^{'}\leq p-1$, together with the descents occurring at the $(2r-2)$-th and $(2r-3)$-th blocks of all paths in $P^{2r}_{l}$. By Corollary \ref{cor rules l}, we obtain the following formulas:
\begin{description}
  \item[\textbf{For $s=2$ and $k = 0$}]: 		 
		\begin{align}\label{s20}
			\mathcal{M}\Big(\lambda\big(4;(x_{(4,0)},x_{(2,0)})\big)\Big) & = \sum_{t^{'}=0}^{p-1}\mathcal{M}\Big(\lambda\big(2;(x_{(2,0)},x_{(1,0)})\big)\Big)+\frac{p-1}{2}(p-1)
		\end{align}
		\item[\textbf{For $s=2$ and $k\geq 1$}]:		
\begin{description}
  \item[\textbf{Case $0\leq l <\frac{p^{k}-1}{2}$}]:
		\begin{equation}\label{s21}
			\begin{aligned}
					\mathcal{M}\Big(\lambda\big(2k+4;(x_{(4,k)},x_{(2,k)}+l)\big)\Big) \\ & \hspace*{-5cm} = \sum_{t^{'}=0}^{p-1}\mathcal{M}\Big(\lambda\big(2k+2;(x_{(2,k)},x_{(1,k)}+(\alpha+p^{k-1}t^{'}))\big)\Big) + \frac{p+1}{2}(p-1)				
			\end{aligned}
		\end{equation}		
\item[\textbf{Case $\frac{p^{k}-1}{2}\leq l <p^{k}$}]:
		\begin{equation}\label{s22}
			\begin{aligned}
				\mathcal{M}\Big(\lambda\big(2k+4;(x_{(4,k)},x_{(2,k)}+l)\big)\Big) \\ & \hspace*{-5.3cm} = \sum_{t^{'}=0}^{p-1}\mathcal{M}\Big(\lambda\big(2k+2;(x_{(2,k)},x_{(1,k)}+(\alpha+p^{k-1}t^{'}))\big)\Big)+ \frac{p-1}{2}(p-1)
			\end{aligned}
		\end{equation}
\end{description}
\item[\textbf{For $s\geq3$ and $k=0$}]:
\begin{equation}\label{k0}
			\begin{aligned}
		\mathcal{M}\Big(\lambda\big(2r;(x_{(2s,0)},x_{(s,0)})\big)\Big)&=p\mathcal{M}\Big(\lambda\big(2r-2; (x_{(2s-2,0)},x_{(s-1,0)})\big)\Big)\\&+p^{s-2}\left(\frac{(p-1)^{2}}{2}\right)+p^{s-3}(p-1)\left(\frac{p(p-1)}{2}\right)
		\end{aligned}
		\end{equation}
	\item[\textbf{For $s\geq3$ and $1\leq k\leq r-1$}]:
\begin{description}
  \item[\textbf{Case $l=0$}]: 
  \begin{equation}\label{l0}
		\begin{aligned}
			\mathcal{M}\Big(\lambda\big(2r;(x_{(2s,k)},x_{(s,k)}+0)\big)\Big) \\ &\hspace*{-3cm}=
				\sum_{t^{'}=0}^{p-1}\mathcal{M}\Big(\lambda\big(2r-2;(x_{(2s-2,k)},x_{(s-1,k)}+p^{k-1}t^{'})\big)\Big) \\ &\hspace*{-3cm}+p^{s-2}\left(\frac{p^{2}-1}{2}\right)+p^{s-3}(p-1)\left(\frac{p(p-1)}{2}\right)
			\end{aligned}
	\end{equation}	
\item[\textbf{Case $1\leq l <\frac{p^{k}-1}{2}$ and $j^{k}_{t-1}< l \leq j^{k}_{t}$, $1\leq t \leq \frac{p-1}{2}~\text{and}~l\neq j^{k}_{\frac{p-1}{2}}$}]:
\begin{equation}\label{less l}
		\begin{aligned}
			\mathcal{M}\Big(\lambda\big(2r;(x_{(2s,k)},x_{(s,k)}+l)\big)\Big) \\ &\hspace*{-3cm} =	\sum_{t^{'}=0}^{p-1}\mathcal{M}\Big(\lambda\big(2r-2;(x_{(2s-2,k)},x_{(s-1,k)}+(\alpha+p^{k-1}t^{'}))\big)\Big) \\ & \hspace*{-3cm}\quad +p^{s-2}\left(\frac{p^{2}-1}{2}\right)+p^{s-3}(p-1)\left(\frac{p(p-1)}{2}+t\right)
		\end{aligned}
	\end{equation}
\item[\textbf{Case $\frac{p^{k}-1}{2}\leq l <p^{k}$ and $j^{k}_{t-1}\leq l \leq j^{k}_{t}$, $\frac{p+1}{2}\leq t \leq p-1$}]:
\begin{equation}\label{greater l}
			\begin{aligned}
				\mathcal{M}\Big(\lambda\big(2r;(x_{(2s,k)},x_{(s,k)}+l)\big)\Big)\\&\hspace*{-3cm}=\sum_{t^{'}=0}^{p-1}\mathcal{M}\Big(\lambda\big(2r-2;(x_{(2s-2,k)},x_{(s-1,k)}+(\alpha+p^{k-1}t^{'}))\big)\Big) \\  &\hspace*{-3cm}\quad +p^{s-2}\left(\frac{(p-1)^{2}}{2}\right) +p^{s-3}(p-1)\left(\frac{p(p-1)}{2}+t\right)
			\end{aligned}
		\end{equation}
\end{description}
\end{description}
		\end{definition}
	\begin{theorem}\label{p0}
		For $r\geq 1$, the $p^{(0)}$-Fibonacci number at the vertex $\lambda\big(2r;(x_{(2s,0)},x_{(s,0)})\big)$ is
		\begin{align*}
			\mathcal{M}\Big(\lambda\big(2;(x_{(2,0)},x_{(1,0)})\big)\Big)&=\frac{p-1}{2}, \\
			\mathcal{M}\Big(\lambda\big(2r;(x_{(2s,0)},x_{(s,0)})\big)\Big)&=\frac{p-1}{2}(2(s-1)p^{s-1}-(2s-3)p^{s-2}),~s\geq 2
		\end{align*}
		where $r=s$.  
	\end{theorem}
	\begin{proof}
		We prove the theorem by mathematical induction on $s$.  
		
		For $s=1$, by Lemma \ref{2p-3}, we have
		$$\mathcal{M}\Big(\lambda\big(2;(x_{(2,0)},x_{(1,0)})\big)\Big) = \frac{p-1}{2}.$$
		
		For $s=2$, again by Lemma \ref{2p-3}, there is no descent at the $1^{st}$ block of any path in the set $P^{4}_{0}$. From equation \eqref{s20}, the $p^{(0)}$-Fibonacci number at the vertex $\lambda\big(4;(x_{(4,0)},x_{(2,0)})\big)$ is
		$$\mathcal{M}\Big(\lambda\big(4;(x_{(4,0)},x_{(2,0)})\big)\Big) =p\left(\frac{p-1}{2}\right)+(p-1)\frac{p-1}{2}=\frac{p-1}{2}(2p-1).$$
		
		Assume that the theorem holds for $s-1$, i.e.,
		\begin{align}\label{s-10}
			\mathcal{M}\Big(\lambda\big(2(r-1);(x_{(2(s-1),0)},x_{(s-1,0)})\big)\Big) =\frac{p-1}{2}(2(s-2)p^{s-2}-(2s-5)p^{s-3}).
		\end{align}
		We now prove that the theorem holds for $s$. By
		substituting the $p^{(0)}$-Fibonacci numbers from equation \eqref{s-10} into equation \eqref{k0} of Definition \ref{Fibo}, we obtain	
		\begin{align*}
			\mathcal{M}\Big(\lambda\big(2r;(x_{(2s,0)},x_{(s,0)})\big)\Big)&=p\left(\frac{p-1}{2}\left(2(s-2)p^{s-2}-(2s-5)p^{s-3}\right)\right)\\&\quad+2p^{s-2}(p-1)\left(\frac{p-1}{2}\right)\\
			&=\left(\frac{p-1}{2}\right)\left(2(s-1)\cdot p^{s-1}-(2s-3)p^{s-2}\right)
		\end{align*}
		Hence, the formula holds for all $s\geq 2$.	
	\end{proof}
		\begin{theorem}\label{p1}
		For $r\geq2$, the $p^{(1)}$-Fibonacci number at the vertex $\lambda\big(2r;(x_{(2s,1)},x_{(s,1)}+l)\big)$, where $l=j^{1}_{t}=t$ with $0\leq l\leq p-1$, and where $l$ denotes the position of the vertex in $V^{2r}_{1}$, as follows
		\begin{align*}
			\mathcal{M}\Big(\lambda\big(4;(x_{(2,1)},x_{(1,1)}+t)\big)\Big) & = \dfrac{p-1}{2}, 0\leq t\leq p-1,
		\end{align*}
		\begin{displaymath}
			\mathcal{M}\Big(\lambda\big(6;(x_{(4,1)},x_{(2,1)}+t)\big)\Big) =
			\begin{cases}
				\dfrac{p-1}{2}(2p+1), & 0\leq t<\dfrac{p-1}{2};
				\\
				\dfrac{p-1}{2}(2p-1), & \dfrac{p-1}{2}\leq t\leq p-1,
			\end{cases}
		\end{displaymath}
		and for $s\geq 3$
		
		$\mathcal{M}\Big(\lambda\big(2r;(x_{(2s,1)},x_{(s,1)}+t)\big)\Big) $
		\begin{displaymath}
			=
			\begin{cases}
				\dfrac{p^{s-3}(p-1)}{2}(2(s-1)p^{2}+2t-(2s-5)), & 0\leq t<\dfrac{p-1}{2};
				\\
				\dfrac{p^{s-3}(p-1)}{2}(2(s-1)p^{2}+2t-2p-(2s-5)), & \dfrac{p-1}{2}\leq t\leq p-1.
			\end{cases}
		\end{displaymath}		
	\end{theorem}
	\begin{proof}
		We prove the theorem by mathematical induction on $s$.
		
		For $s=1$, by Lemma $\ref{2p-3}$, we have
		\begin{align} \label{s14}
			\mathcal{M}\Big(\lambda\big(4;(x_{(2,1)},x_{(1,1)}+t)\big)\Big) =\frac{p-1}{2},~\forall~0\leq t\leq p-1.
		\end{align}
		
		For $s=2$, again by Lemma $\ref{2p-3}$, there is no descent at the $3^{rd}$ block. Substituting the $p^{(1)}$-Fibonacci numbers from equation \eqref{s14} into equation \eqref{s21} of Definition \ref{Fibo}, we obtain for $0\leq t<\frac{p-1}{2},$
		\begin{align}\label{4p-5 i}
			\mathcal{M}\Big(\lambda\big(6;(x_{(4,1)},x_{(2,1)}+t)\big)\Big) &=p\left(\frac{p-1}{2}\right)+\left(\frac{p+1}{2}\right)(p-1)  =\frac{p-1}{2}\left(2p+1\right)
		\end{align}
		
		Similarly, by substituting the $p^{(1)}$-Fibonacci numbers from equation \eqref{s14} into equation \eqref{s22} of Definition \ref{Fibo}, we obtain for $\frac{p-1}{2}\leq t \leq p-1,$
		\begin{align}\label{4p-5 ii}
			\mathcal{M}\Big(\lambda\big(6;(x_{(4,1)},x_{(2,1)}+t)\big)\Big) &=p\left(\frac{p-1}{2}\right)+\left(\frac{p-1}{2}\right)(p-1)=\frac{p-1}{2}\left(2p-1\right)
		\end{align}
		
		For $s=3$, substituting the $p^{(1)}$-Fibonacci numbers from equations \eqref{4p-5 i} and \eqref{4p-5 ii} into equation \eqref{less l} of Definition \ref{Fibo}, we obtain for $0\le t<\frac{p-1}{2},$
		\begin{align*}
		\mathcal{M}\Big(\lambda\big(8;(x_{(6,1)},x_{(3,1)}+t)\big)\Big) \\	&\hspace*{-4cm}=\left(\frac{p-1}{2}\right)\left(\frac{(p-1)(2p+1)}{2}+\frac{(p+1)(2p-1)}{2}+p(p+1)+p(p-1)+2t\right) \\ &\hspace*{-4cm}=\left(\frac{p-1}{2}\right)(4p^{2}+2t-1).
		\end{align*}
		
		Similarly, substituting the $p^{(1)}$-Fibonacci numbers from equations \eqref{4p-5 i} and \eqref{4p-5 ii} into equation \eqref{greater l} of Definition \ref{Fibo}, we obtain, for $\frac{p-1}{2}\leq t\leq p-1,$
		\begin{align*}
			\mathcal{M}\Big(\lambda\big(8;(x_{(6,1)},x_{(3,1)}+t)\big)\Big)
			\\ &\hspace*{-4cm}=\left(\frac{p-1}{2}\right)\left(\frac{(p-1)(2p+1)}{2}+\frac{(p+1)(2p-1)}{2}+p(p-1)+p(p-1)+2t\right) \\&\hspace*{-4cm}=\left(\frac{p-1}{2}\right)(4p^{2}+2t-2p-1).
		\end{align*}
		
		Assume that the theorem holds for $s-1$. That is, for $0\leq t<\frac{p-1}{2},$
		\begin{align}\label{p,s-1}
			\mathcal{M}\Big(\lambda\big(2(r-1);(x_{(2(s-1),1)},x_{(s-1,1)}+t)\big)\Big) & = \dfrac{p^{s-4}(p-1)}{2}\left(2(s-2)p^{2}+2t-(2s-7)\right). 
		\end{align}
		
		For $\frac{p-1}{2}\leq t\leq p-1,$
		\begin{equation}\label{p,s-11}
			\begin{aligned}
					\mathcal{M}\Big(\lambda\big(2(r-1);(x_{(2(s-1),1)},x_{(s-1,1)}+t)\big)\Big) \\&\hspace*{-2.5cm} = 	\dfrac{p^{s-4}(p-1)}{2}\left(2(s-2)p^{2}+2t-2p-(2s-7)\right). 		\end{aligned}
		\end{equation}	
		We now show that the theorem holds for $s$. 
		
		Substituting the $p^{(1)}$-Fibonacci numbers from equations \eqref{p,s-1} and \eqref{p,s-11} into equation \eqref{less l} of Definition \ref{Fibo}, we obtain, for $0\leq t <\frac{p-1}{2},$
		
		$\mathcal{M}\Big(\lambda\big(2r;(x_{(2s,1)},x_{(s,1)}+t)\big)\Big)$
		\begin{align*}
			&= \sum_{t^{'}=0}^{p-1} \mathcal{M}\Big(\lambda\big(2(r-1);(x_{(2(s-1),1)},x_{(s-1,1)}+t^{'})\big)\Big)+p^{s-2}\left(\frac{p^{2}-1}{2}\right)\\ &\quad+p^{s-3}(p-1)\left(\frac{p(p-1)}{2}+t\right) \\ \\ &= \left(\frac{p^{s-4}(p-1)}{2}\right)\left(\sum\limits_{t^{'}=0}^{\frac{p-3}{2}}2(s-2)p^{2}+2t^{'}-(2s-7)+p^{2}(p-1)\left(\frac{p+1}{2}\right)\right. \\ & \left.\quad+\sum\limits_{t^{'}=\frac{p-1}{2}}^{p-1}2(s-2)p^{2}+2t^{'}-2p-(2s-7)+p(p-1)\left(\frac{p(p-1)}{2}+t\right)\right) \\ \\
			&= \frac{p^{s-3}(p-1)}{2}\left(2(s-1)p^{2}+2t-(2s-5)\right)
		\end{align*}
		
		Similarly, substituting the $p^{(1)}$-Fibonacci numbers from equations \eqref{p,s-1} and \eqref{p,s-11} into equation \eqref{greater l} of Definition \ref{Fibo}, we obtain, for $\frac{p-1}{2}\leq t\leq p-1$, 
		\begin{align*}
			 \mathcal{M}\Big(\lambda\big(2r;(x_{(2s,1)},x_{(s,1)}+t)\big)\Big)&= \sum_{t^{'}=0}^{p-1} \mathcal{M}\Big(\lambda\big(2(r-1);(x_{(2(s-1),1)},x_{(s-1,1)}+t^{'})\big)\Big)\\&\quad+p^{s-2}\left(\frac{(p-1)^{2}}{2}\right)+p^{s-3}(p-1)\left(\frac{p(p-1)}{2}+t\right) \\ 
			&= \frac{p^{s-3}(p-1)}{2}(2(s-1)p^{2}+2t-2p-(2s-5)).
		\end{align*}
		
		Thus, by the principle of mathematical induction, the result holds for all $s\geq 3$.
	\end{proof}
	\begin{example}
		The $5^{1}$-Fibonacci numbers at the vertices $\lambda\big(2(s+1);(x_{(2s,1)},x_{(s,1)}+l)\big)$, where $3\leq s\leq 9$, are given by
		\begin{enumerate}
			\item For $l=0$, we have 
			
			$\{\mathcal{M}\Big(\lambda\big(2(s+1);(x_{(2s,1)},x_{(s,1)}+0)\big)\Big)|~3\leq s\leq 10\}$ 
			
			\quad\quad\quad\quad\quad\quad $=\{198, 1470, 9750, 60750, 363750, 2118750, 12093750\}$.
			\item For $l=1$, we have 
			
			$\{\mathcal{M}\Big(\lambda\big(2(s+1);(x_{(2s,1)},x_{(s,1)}+1)\big)\Big)|~3\leq s\leq 10\} $
			
			\quad\quad\quad\quad\quad\quad$=\{202, 1490, 9850, 61250, 366250, 2131250, 12156250\}$.
			\item For $l=2$, we have 
			
			$\{\mathcal{M}\Big(\lambda\big(2(s+1);(x_{(2s,1)},x_{(s,1)}+2)\big)\Big)|~3\leq s\leq 10\}$
			
			\quad\quad\quad\quad\quad\quad$=\{186, 1410, 9450, 59250, 356250, 2081250, 11906250\}$.
			\item For $l=3$, we have 
			
			$\{\mathcal{M}\Big(\lambda\big(2(s+1);(x_{(2s,1)},x_{(s,1)}+3)\big)\Big)|~3\leq s\leq 10\}$
			
			\quad\quad\quad\quad\quad\quad$=\{190, 1430, 9550, 59750, 358750, 2093750, 11968750\}$.
			\item For $l=4$, we have 
			
			$\{\mathcal{M}\Big(\lambda\big(2(s+1);(x_{(2s,1)},x_{(s,1)}+4)\big)\Big)|~3\leq s\leq 10\}$
			
			\quad\quad\quad\quad\quad\quad$=\{194, 1450, 9650, 60250, 361250, 2106250, 12031250\}$.
		\end{enumerate}
	\end{example}
	\begin{theorem} \label{base}
		Let $2\leq k\leq r-1$ and $1\leq s\leq 3$ $(\text{where}~r-k=s)$. Then the $p^{(k)}$-Fibonacci number at the vertex $\lambda\big(2r;(x_{(2s,k)},x_{(s,k)}+l)\big)$, $0\leq l <p^{k}$, is given by 
		\begin{description}
		  \item[For $s=1:$] \begin{align*}
			\mathcal{M}\Big(\lambda\big(2k+2;(x_{(2,k)},x_{(1,k)}+l)\big)\Big) &=\frac{p-1}{2},~\forall~0\leq l\leq p^{k}-1.
		\end{align*}
		  \item[For $s=2:$] \begin{displaymath}
			\mathcal{M}\Big(\lambda\big(2k+4;(x_{(4,k)},x_{(2,k)}+l)\big)\Big)  =
			\begin{cases}
				\frac{p-1}{2}\left(2p+1\right), & \text{for $0\leq l \leq \frac{p^{k}-1}{2}$;} \\ 
				\frac{p-1}{2}\left(2p-1\right), & \text{for $\frac{p^{k}-1}{2}\leq l <p^{k}$.}
			\end{cases}	
		\end{displaymath}
		  \item[For $s\geq3:$] $\mathcal{M}\Big(\lambda\big(2k+6;(x_{(6,k)},x_{(3,k)}+l)\big)\Big) $
		\begin{displaymath}
			 =
			\begin{cases}
				\frac{p-1}{2}\left(4p^{2}+1\right), & \text{for $l =0$;} \\ 
				\frac{p-1}{2}\left(4p^{2}+2t-1\right), & \text{for $l\in [(t-1)\sum\limits_{\iota=0}^{k-1}p^{\iota}+1,\ldots,t\sum\limits_{\iota=0}^{k-1}p^{\iota}],~1\leq t\leq \frac{p-3}{2}$;} \\ 
				\frac{p-1}{2}\left(4p^{2}+1+p-1\right), & \text{for $l\in [\frac{p-3}{2}\sum\limits_{\iota=0}^{k-1}p^{\iota}+1,\ldots,\frac{p-1}{2}\sum\limits_{\iota=1}^{k-1}p^{\iota}-1]$;}\\ 
				\frac{p-1}{2}\left(4p^{2}-1+p-1\right), & \text{for $l\in [\frac{p-1}{2}\sum\limits_{\iota=1}^{k-1}p^{\iota},\ldots,\frac{p-1}{2}\sum\limits_{\iota=0}^{k-1}p^{\iota}-1]$;}\\ 
				\frac{p-1}{2}\left(4p^{2}-1-p-1\right), & \text{for $l=\frac{p-1}{2}\sum\limits_{\iota=0}^{k-1}p^{\iota}$;}\\ 
				\frac{p-1}{2}\left(4p^{2}-1+2(t-p)\right), & \text{for } l\in[(t-1)\sum\limits_{\iota=0}^{k-1}p^{\iota}+1,\ldots,t\sum\limits_{\iota=0}^{k-1}p^{\iota}],\\ & ~\frac{p+1}{2}\leq t\leq p-1.\\			\end{cases}	
		\end{displaymath}
		\end{description}		
	\end{theorem}
	\begin{proof}
		 \textbf{Case 1: For $s=1$}
		
		By Lemma $\ref{2p-3}$, we have
		\begin{align}\label{sk1}
			\mathcal{M}\Big(\lambda\big(2k+2;(x_{(2,k)},x_{(1,k)}+l)\big)\Big) =\frac{p-1}{2},~\forall~0\leq l\leq p^{k}-1 
		\end{align}
		
		\textbf{Case 2: For $s=2$}
		
		By Lemma $\ref{2p-3}$, there is no descent at the $(2k+1)$-th block. Substituting the $p^{(k)}$-Fibonacci numbers from equation \eqref{sk1} into equation \eqref{s21} of Definition \ref{Fibo}, we obtain, for $0\leq l<\frac{p^{k}-1}{2},$
		\begin{equation}\label{4p-5 i(1)}
			\begin{aligned}
				\mathcal{M}\Big(\lambda\big(2k+4;(x_{(4,k)},x_{(2,k)}+l)\big)\Big) & = \sum_{t^{'}=0}^{p-1}\mathcal{M}\Big(\lambda\big(2k+2;(x_{(2,k)},x_{(1,k)}+(\alpha+p^{k-1}t^{'}))\big)\Big) \\ &\quad + \frac{p+1}{2}(p-1) 	=\frac{p-1}{2}\left(2p+1\right).
			\end{aligned}
		\end{equation}
		Similarly, by substituting the $p^{(k)}$-Fibonacci numbers from equation \eqref{sk1} into equation \ref{s22} of Definition \ref{Fibo}, we obtain, for $\frac{p^{k}-1}{2}\leq l \leq p^{k}-1,$
		\begin{equation}\label{4p-5 ii(1)}
		\begin{aligned}
			\mathcal{M}\Big(\lambda\big(2k+4;(x_{(4,k)},x_{(2,k)}+l)\big)\Big) & = \sum_{t^{'}=0}^{p-1}\mathcal{M}\Big(\lambda\big(2k+2;(x_{(2,k)},x_{(1,k)}+(\alpha+p^{k-1}t^{'}))\big)\Big) \\ & \quad + \frac{p-1}{2}(p-1) =\frac{p-1}{2}\left(2p-1\right).
		\end{aligned}
		\end{equation}
	
\textbf{Case 3: For $s=3$}
		
		We obtain the $p^{(k)}$-Fibonacci number at the vertex $\mathcal{M}\Big(\lambda\big(2k+6;(x_{(6,k)},x_{(3,k)}+0)\big)\Big)$ by substituting the $p^{(k)}$-Fibonacci numbers from equations \eqref{4p-5 i(1)} and \eqref{4p-5 ii(1)} into equation \eqref{l0} of Definition \ref{Fibo}. For $l=0$, we get 
		\begin{align*}
			\mathcal{M}\Big(\lambda\big(2k+6;(x_{(6,k)},x_{(3,k)}+0)\big)\Big) \\ &\hspace*{-4cm}=\left(\frac{p-1}{2}\right)\left(\frac{(p+1)(2p+1)}{2}+\frac{(p-1)(2p-1)}{2}+p(p+1)+p(p-1)\right) \\ &\hspace*{-4cm}=\left(\frac{p-1}{2}\right)(4p^{2}+1).\\
		\end{align*}
		
		Similarly, substituting the $p^{(k)}$-Fibonacci numbers from equations \eqref{4p-5 i(1)} and \eqref{4p-5 ii(1)} into equation \eqref{less l} of Definition \ref{Fibo}, we obtain, for $1\leq t\leq \frac{p-3}{2}$ and $l\in [(t-1)\sum\limits_{\iota=0}^{k-1}p^{\iota}+1,\ldots, t\sum\limits_{\iota=0}^{k-1}p^{\iota}]$
		\begin{align*}
			\mathcal{M}\Big(\lambda\big(2k+6;(x_{(6,k)},x_{(3,k)}+l)\big)\Big)\\ &\hspace*{-4cm}=\left(\frac{p-1}{2}\right)\left(\frac{(p+1)(2p+1)}{2}+\frac{(p-1)(2p-1)}{2}+p(p+1)+p(p-1)+2t\right) \\ &\hspace*{-4cm}=\left(\frac{p-1}{2}\right)(4p^{2}+2t-1).\\
		\end{align*}
		
		Next, substituting the $p^{(k)}$-Fibonacci numbers from equations \eqref{4p-5 i(1)} and \eqref{4p-5 ii(1)} into equation \eqref{less l} of Definition \ref{Fibo}, we obtain, for $l\in [\frac{p-3}{2}\sum\limits_{\iota=0}^{k-1}p^{\iota}+1,\ldots,\frac{p-1}{2}\sum\limits_{\iota=1}^{k-1}p^{\iota}-1],$
		\begin{align*}
		\mathcal{M}\Big(\lambda\big(2k+6;(x_{(6,k)},x_{(3,k)}+l)\big)\Big)\\ 	&\hspace*{-5.5cm}=\left(\frac{p-1}{2}\right)\left(\frac{(p+1)(2p+1)}{2}+\frac{(p-1)(2p-1)}{2}+p(p+1)+p(p-1)+2\left(\frac{p-1}{2}\right)\right) \\ &\hspace*{-5.5cm}=\left(\frac{p-1}{2}\right)(4p^{2}+1+p-1).\\ \\
		\end{align*}
		
		Similarly, substituting the $p^{(k)}$-Fibonacci numbers from equations \eqref{4p-5 i(1)} and \eqref{4p-5 ii(1)} into equation \eqref{less l} of Definition \ref{Fibo}, we obtain, for $l\in [\frac{p-1}{2}\sum\limits_{\iota=1}^{k-1}p^{\iota},\ldots,\frac{p-1}{2}\sum\limits_{\iota=0}^{k-1}p^{\iota}-1],$
		
		$\mathcal{M}\Big(\lambda\big(2k+6;(x_{(6,k)},x_{(3,k)}+l)\big)\Big)$
		\begin{align*}
			&=\left(\frac{p-1}{2}\right)\left(\frac{(p-1)(2p+1)}{2}+\frac{(p+1)(2p-1)}{2}+p(p+1)+p(p-1)+2\left(\frac{p-1}{2}\right)\right) \\ &=\left(\frac{p-1}{2}\right)(4p^{2}-1+p-1).
		\end{align*}
		
		Likewise, substituting the $p^{(k)}$-Fibonacci numbers from equations \eqref{4p-5 i(1)} and \eqref{4p-5 ii(1)} into equation \eqref{greater l} of Definition \ref{Fibo}, we obtain, for $l=\frac{p-1}{2}\sum\limits_{\iota=0}^{k-1}p^{\iota},$
		
		$\mathcal{M}\Big(\lambda\big(2k+6;(x_{(6,k)},x_{(3,k)}+l)\big)\Big)$
		\begin{align*}
			&=\left(\frac{p-1}{2}\right)\left(\frac{(p-1)(2p+1)}{2}+\frac{(p+1)(2p-1)}{2}+p(p-1)+p(p-1)+2\left(\frac{p-1}{2}\right)\right) \\ &=\left(\frac{p-1}{2}\right)(4p^{2}-1-p-1).
		\end{align*}
		
		Substituting the $p^{(k)}$-Fibonacci numbers from equations \eqref{4p-5 i(1)} and \eqref{4p-5 ii(1)} into equation \eqref{greater l} of Definition \ref{Fibo}, we obtain, for $\frac{p+1}{2}\leq t\leq p-1$ and $l\in [(t-1)\sum\limits_{\iota=0}^{k-1}p^{\iota}+1,\ldots,t\sum\limits_{\iota=0}^{k-1}p^{\iota}],$
		
		$\mathcal{M}\Big(\lambda\big(2k+6;(x_{(6,k)},x_{(3,k)}+l)\big)\Big)$
		\begin{align*}
			&=\left(\frac{p-1}{2}\right)\left(\frac{(p-1)(2p+1)}{2}+\frac{(p+1)(2p-1)}{2}+p(p-1)+p(p-1)+2t\right) \\&=\left(\frac{p-1}{2}\right)(4p^{2}-1+2t-2p).
		\end{align*}
		
	\end{proof}
	
	\begin{theorem} \label{less}
		For $2\leq k\leq r-1$ and $3\leq s< k+2$ $(\text{where}~r-k=s)$, the $p^{(k)}$-Fibonacci number at the vertex $\lambda\big(2r;(x_{(2s,k)},x_{(s,k)}+l)\big)$, where $0\leq l <p^{k}$ and $l$ corresponds to the position of the vertex in $V^{2r}_{k}$ is
		\begin{description}
			\item[(a)]  If $l=0$, then $$\mathcal{M}\Big(\lambda\big(2r;(x_{(2s,k)},x_{(s,k)}+l)\big)\Big) = \frac{p-1}{2}(2(s-1)p^{s-1}+1)$$
			\item[(b)] For $1\leq t\leq \frac{p-3}{2}$, 
			\begin{enumerate}
				\item[(i)] If $l\in\left[(t-1)\sum\limits_{\iota=i}^{k-1}p^{\iota}+p^{i},\ldots,(t-1)\sum\limits_{\iota=i+1}^{k-1}p^{\iota}+p^{i+1}-1\right],~0\leq i \leq s-4,~s\neq3,$
				then
				\begin{align*}
					\mathcal{M}\Big(\lambda\big(2r;(x_{(2s,k)},x_{(s,k)}+l)\big)\Big) \\&\hspace*{-3cm}=\frac{p-1}{2}\Big(2(s-1)p^{s-1}+1+2t\sum_{j=0}^{s-3}p^{j}
					-2\sum_{\mu=0}^{s-i-4}p^{\mu}\Big)
				\end{align*} 
				\item [(ii)] If $l\in\left[(t-1)\sum\limits_{\iota=s-3}^{k-1}p^{\iota}+p^{s-3},\ldots,t\sum\limits_{\iota=0}^{k-1}p^{\iota}\right]$, then 
				\begin{align*}
					\mathcal{M}\Big(\lambda\big(2r;(x_{(2s,k)},x_{(s,k)}+l)\big)\Big) = 	&\frac{p-1}{2}\Big(2(s-1)p^{s-1}+1+2t\sum_{j=0}^{s-3}p^{j}\Big)
				\end{align*}
			\end{enumerate} 
			\item [(c)] For $t=\frac{p-1}{2},$
			\begin{enumerate}
				\item [(i)] If $l\in\left[\frac{p-3}{2}\sum\limits_{\iota=i}^{k-1}p^{\iota}+p^{i},\ldots,\frac{p-3}{2}\sum\limits_{\iota=i+1}^{k-1}p^{\iota}+p^{i+1}-1\right],~0\leq i \leq s-4,~s\neq3,$
				then
				\begin{align*}
					\mathcal{M}\Big(\lambda\big(2r;(x_{(2s,k)},x_{(s,k)}+l)\big)\Big) \\ &\hspace*{-1.5cm}=\frac{p-1}{2}\Big(2(s-1)p^{s-1}+1+(p-1)\sum_{j=0}^{s-3}p^{j}
					-2\sum_{\mu=0}^{s-i-4}p^{\mu}\Big)
				\end{align*} 
				\item [(ii)]  If $l\in\left[\frac{p-3}{2}\sum\limits_{\iota=s-3}^{k-1}p^{\iota}+p^{s-3},\ldots,\frac{p-1}{2}\sum\limits_{\iota=s-2}^{k-1}p^{\iota}-1\right]$, then 
				$$\mathcal{M}\Big(\lambda\big(2r;(x_{(2s,k)},x_{(s,k)}+l)\big)\Big) = \frac{p-1}{2}(2(s-1)p^{s-1}+1+p^{s-2}-1)$$ \label{(d)}
				\item [(iii)] If $l\in\left[\frac{p-1}{2}\sum\limits_{\iota=i^{'}}^{k-1}p^{\iota},\ldots,\frac{p-1}{2} \sum\limits_{\iota=i^{'}-1}^{k-1}p^{\iota}-1\right],~1\leq i^{'}\leq s-2$, then 
				\begin{align*}
					\mathcal{M}\Big(\lambda\big(2r;(x_{(2s,k)},x_{(s,k)}+l)\big)\Big) \\
					 &\hspace*{-3.5cm}= \frac{p-1}{2}\Big(2(s-1)p^{s-1}-1-p^{s-2}-2\sum_{\theta = 1}^{s-3}p^{\theta}+2\sum_{\mu =s-i^{'}-1}^{s-2}p^{\mu}-1\Big)
				\end{align*}
				\label{(f)}
				\item [(iv)] If $l=\frac{p-1}{2} \sum\limits_{\iota=0}^{k-1}p^{\iota}$, then 
				\begin{align*}
					\hspace*{-3cm}\mathcal{M}\Big(\lambda\big(2r;(x_{(2s,k)},x_{(s,k)}+l)\big)\Big) \\&\hspace*{-2cm}= \frac{p-1}{2}\Big(2(s-1)p^{s-1}-1-p^{s-2}-2\sum_{\theta = 1}^{s-3}p^{\theta}-1\Big)
				\end{align*}
			\end{enumerate}
			
			\item [(d)] For $\frac{p+1}{2}\leq t\leq p-1,$
			\begin{enumerate}
				\item [(i)] 	If $l\in\left[(t-1)\sum\limits_{\iota=i}^{k-1}p^{\iota}+p^{i},\ldots,(t-1)\sum\limits_{\iota=i+1}^{k-1}p^{\iota}+p^{i+1}-1\right],~0\leq i \leq s-4,~s\neq3$, then
				
				$\mathcal{M}\Big(\lambda\big(2r;(x_{(2s,k)},x_{(s,k)}+l)\big)\Big)$
				\begin{align*}
					=&\frac{p-1}{2}\Big(2(s-1)p^{s-1}-1+(2t-2p)\sum_{j=0}^{s-3}p^{j}
					-2\sum_{\mu=0}^{s-i-4}p^{\mu}\Big)
				\end{align*} \label{(b)}
				\item [(ii)] If $l\in\left[(t-1)\sum\limits_{\iota=s-3}^{k-1}p^{\iota}+p^{s-3},\ldots,t\sum\limits_{\iota=0}^{k-1}p^{\iota}\right]$, then 
				\begin{align*}
					\mathcal{M}\Big(\lambda\big(2r;(x_{(2s,k)},x_{(s,k)}+l)\big)\Big) = 	 &\frac{p-1}{2}\Big(2(s-1)p^{s-1}-1+(2t-2p)\sum_{j=0}^{s-3}p^{j}
					\Big)
				\end{align*}\label{(c)}
			\end{enumerate}
		\end{description}
	\end{theorem}
	\begin{proof}
		We prove the theorem by mathematical induction on $s$.
		
		For $s=3$, the given formula holds by Theorem \ref{base}.
		
		By the induction hypothesis, we assume that the theorem holds for $s-1$, where $3\leq s-1<k+1;$ that is, 
		\begin{enumerate}
			\item  When $l=0$, we have $$\mathcal{M}\Big(\lambda\big(2(r-1);(x_{(2(s-1),k)},x_{(s-1,k)})\big)\Big) = \frac{p-1}{2}(2(s-2)p^{s-2}+1)$$
			
			\item (i) For $l\in\left[(t-1)\sum\limits_{\iota=i}^{k-1}p^{\iota}+p^{i},\ldots,(t-1)\sum\limits_{\iota=i+1}^{k-1}p^{\iota}+p^{i+1}-1\right],~0\leq i \leq s-5,~s-1\neq3$, $1\leq t\leq \frac{p-3}{2}$, we have
			\begin{align*}
				\mathcal{M}\Big(\lambda\big(2(r-1);(x_{(2(s-1),k)},x_{(s-1,k)}+l)\big)\Big) \\ &\hspace*{-4.5cm}=\frac{p-1}{2}\Big(2(s-2)p^{s-2}+1+2t\sum_{j=0}^{s-4}p^{j}
				-2\sum_{\mu=0}^{s-i-5}p^{\mu}\Big)
			\end{align*} 
			(ii) For $l\in\left[(t-1)\sum\limits_{\iota=s-4}^{k-1}p^{\iota}+p^{s-4},\ldots,t\sum\limits_{\iota=0}^{k-1}p^{\iota}\right],~ 1\leq t\leq \frac{p-3}{2}$, we have
			\begin{align*}
				\mathcal{M}\Big(\lambda\big(2(r-1);(x_{(2(s-1),k)},x_{(s-1,k)}+l)\big)\Big) = 	&\frac{p-1}{2}\Big(2(s-2)p^{s-2}+1+2t\sum_{j=0}^{s-4}p^{j}\Big)
			\end{align*}
			
			\item For $t=\frac{p-1}{2}:$
			\begin{enumerate}
				\item For $l\in\left[\frac{p-3}{2}\sum\limits_{\iota=i}^{k-1}p^{\iota}+p^{i},\ldots,\frac{p-3}{2}\sum\limits_{\iota=i+1}^{k-1}p^{\iota}+p^{i+1}-1\right],~0\leq i \leq s-5,~s-1\neq3$, we have
				
				$\mathcal{M}\Big(\lambda\big(2(r-1);(x_{(2(s-1),k)},x_{(s-1,k)}+l)\big)\Big)$
				\begin{align*}
					 =&\frac{p-1}{2}\Big(2(s-2)p^{s-2}+1+(p-1)\sum_{j=0}^{s-4}p^{j}
					-2\sum_{\mu=0}^{s-i-5}p^{\mu}\Big)\\
				\end{align*}
				\item For  $l\in\left[\frac{p-3}{2}\sum\limits_{\iota=s-4}^{k-1}p^{\iota}+p^{s-4},\ldots,\frac{p-1}{2}\sum\limits_{\iota=s-3}^{k-1}p^{\iota}-1\right]$, we have 
				$$\mathcal{M}\Big(\lambda\big(2(r-1);(x_{(2(s-1),k)},x_{(s-1,k)}+l)\big)\Big) = \frac{p-1}{2}(2(s-2)p^{s-2}+1+p^{s-3}-1)$$
				\item If $l\in\left[\frac{p-1}{2}\sum\limits_{\iota=i^{'}}^{k-1}p^{\iota},\ldots,\frac{p-1}{2} \sum\limits_{\iota=i^{'}-1}^{k-1}p^{\iota}-1\right],~1\leq i^{'}\leq s-3$, then 
				
				$\mathcal{M}\Big(\lambda\big(2(r-1);(x_{(2(s-1),k)},x_{(s-1,k)}+l)\big)\Big)$ 
				\begin{align*}
					&= \frac{p-1}{2}\Big(2(s-2)p^{s-2}-1-p^{s-3}-2\sum_{\theta = 1}^{s-4}p^{\theta}+2\sum_{\mu =s-i^{'}-2}^{s-3}p^{\mu}-1\Big)
				\end{align*}
				\item If $l=\frac{p-1}{2} \sum\limits_{\iota=0}^{k-1}p^{\iota}$, then 
				
				$\mathcal{M}\Big(\lambda\big(2(r-1);(x_{(2(s-1),k)},x_{(s-1,k)}+l)\big)\Big)$  
				\begin{align*}
					= \frac{p-1}{2}\Big(2(s-2)p^{s-2}-1-p^{s-3}-2\sum_{\theta = 1}^{s-4}p^{\theta}-1\Big)
				\end{align*}
			\end{enumerate}
			\item For $\frac{p+1}{2}\leq t\leq p-1:$
			\begin{enumerate}[label=(\roman*)]
				\item For $l\in\left[(t-1)\sum\limits_{\iota=i}^{k-1}p^{\iota}+p^{i},\ldots,(t-1)\sum\limits_{\iota=i+1}^{k-1}p^{\iota}+p^{i+1}-1\right],~0\leq i \leq s-5,~s-1\neq3$, we have
				
				$\mathcal{M}\Big(\lambda\big(2(r-1);(x_{(2(s-1),k)},x_{(s-1,k)}+l)\big)\Big)$
				\begin{align*}
					=&\frac{p-1}{2}\Big(2(s-2)p^{s-2}-1+(2t-2p)\sum_{j=0}^{s-4}p^{j}
					-2\sum_{\mu=0}^{s-i-5}p^{\mu}\Big)
				\end{align*} 
				\item For $l\in\left[(t-1)\sum\limits_{\iota=s-4}^{k-1}p^{\iota}+p^{s-4},\ldots,t\sum\limits_{\iota=0}^{k-1}p^{\iota}\right]$, we have
				
				$\mathcal{M}\Big(\lambda\big(2(r-1);(x_{(2(s-1),k)},x_{(s-1,k)}+l)\big)\Big)$
				\begin{align*}
					 = 	 &\frac{p-1}{2}\Big(2(s-2)p^{s-2}-1+(2t-2p)\sum_{j=0}^{s-4}p^{j}
					\Big)
				\end{align*}
			\end{enumerate}
		\end{enumerate}
		
		We now prove that the theorem holds for $s$. 
		\begin{description}
			\item[(a)] When $l=0$, by Definition \ref{comp}, the components of the vertex $\lambda\big(2r;(x_{(2s,k)},x_{(s,k)}+0)\big)$ are given by $\lambda\big(2(r-1);(x_{(2(s-1),k)},x_{(s-1,k)}+p^{k-1}t^{'})\big)$, where $0\leq t^{'}\leq p-1$. Using equation \eqref{l0} of Definition \ref{Fibo} and substituting the values of $\mathcal{M}\Big(\lambda\big(2(r-1);(x_{(2(s-1),k)},x_{(s-1,k)}+p^{k-1}t^{'})\big)\Big)$, where $0\leq t^{'}\leq p-1$, we obtain
			$$\mathcal{M}\Big(\lambda\big(2r;(x_{(2s,k)},x_{(s,k)})\big)\Big)	= \frac{p-1}{2}\left(2(s-1)\cdot p^{s-1}+1\right) $$
			where
			
			$\mathcal{M}\Big(\lambda\big(2(r-1);(x_{(2(s-1),k)},x_{(s-1,k)}+p^{k-1}t^{'})\big)\Big)$ 
			\begin{displaymath}
				=
				\begin{cases}
					\frac{p-1}{2}\left(2(s-2)p^{s-2}+1\right), & \text{when $t^{'}=0$;}\\
					\frac{p-1}{2}\big(2(s-2)p^{s-2}+1+2t^{'}\sum\limits_{j=0}^{s-4}p^{j}\big) , & \text{when $1\leq t^{'}\leq \frac{p-3}{2}$;} \\	\frac{p-1}{2}\left(2(s-2)p^{s-2}+1+p^{s-3}-1\right), & \text{when  $t^{'} = \frac{p-1}{2}$;} \\ \frac{p-1}{2}\big(2(s-2)p^{s-2}+1+(2t^{'}-2p)\sum\limits_{j=0}^{s-4}p^{j}\big), & \text{when $\frac{p+1}{2}\leq t^{'}\leq p-1$.}
				\end{cases}	
			\end{displaymath}
			\item[(b)]\label{bi}
			\begin{enumerate}
				\item [(i)] Let $l\in\left[(t-1)\sum\limits_{\iota=i}^{k-1}p^{\iota}+p^{i},\ldots,(t-1)\sum\limits_{\iota=i+1}^{k-1}p^{\iota}+p^{i+1}-1\right]$, where $0\leq i \leq s-4,~s\neq3$, and $1\leq t\leq \frac{p-3}{2}$. Then we compute the $p^{(k)}$-Fibonacci number at the vertex $\lambda\big(2r;(x_{(2s,k)},x_{(s,k)}+l)\big)$ as follows:
				
				For $l\in [1,\ldots,p-1]$, the components of the vertex $\lambda\big(2r;(x_{(2s,k)},x_{(s,k)}+l)\big)$ are the same as those of the vertex $\lambda\big(2r;(x_{(2s,k)},x_{(s,k)}+0)\big)$. Substituting the values of the $p^{(k)}$-Fibonacci numbers of the corresponding components into equation \eqref{less l} of Definition \ref{Fibo}, we obtain
				\begin{align*}
					\hspace*{-1cm}\mathcal{M}\Big(\lambda\big(2r;(x_{(2s,k)},x_{(s,k)}+l)\big)\Big)	\\ &\hspace*{-2cm}=  \frac{p-1}{2}\Big(2(s-1)\cdot p^{s-1}+1+2\sum\limits_{j=0}^{s-3}p^{j}-2\sum\limits_{\mu=0}^{s-4}p^{j}\Big)
				\end{align*}
				Next, for $l\in\left[(t-1)\sum\limits_{\iota=i}^{k-1}p^{\iota}+p^{i},\ldots,(t-1)\sum\limits_{\iota=i+1}^{k-1}p^{\iota}+p^{i+1}-1\right],~0\leq i \leq s-4$, $1\leq t\leq \frac{p-3}{2}$ with the additional condition $i\neq 0$ when $t=1$, the $p^{(k)}$-Fibonacci numbers of the components are given as follows:
				
				Let $l=(t-1)\sum\limits_{\iota=i}^{k-1}p^{\iota}+p^{i}$. Then by Definition \ref{comp} the integers corresponding to the components are 
				\begin{align}\label{proj}
					\alpha+p^{k-1}t^{'} = (t-1)\sum\limits_{\iota=i-1}^{k-2}p^{\iota}+p^{i-1}+ p^{k-1}t^{'}
				\end{align}
				where $\alpha = (t-1)\sum\limits_{\iota=i-1}^{k-2}p^{\iota}+p^{i-1}$.
				
				When $0\leq t^{'}\leq t-2$, the equation \eqref{proj} can be rewritten as
				\begin{align}
					\alpha+p^{k-1}t^{'} = t^{'}\sum\limits_{\iota=i-1}^{k-1}p^{\iota}+p^{i-1}+ (t-1-t^{'})\sum\limits_{\iota=i-1}^{k-2}p^{\iota}
				\end{align} 
				Observe that $$(p-t^{'})\sum_{\iota=i-1}^{s-5}p^{\iota}<(t-1-t^{'})\sum\limits_{\iota=i-1}^{k-2}p^{\iota}<(p-t^{'})\sum_{\iota=i-1}^{s-5}p^{\iota}+t^{'}\sum_{\iota=0}^{s-5}p^{\iota}+\sum_{\iota=s-5}^{k-1}p^{\iota}$$
				where $(p-t^{'})\sum\limits_{\iota=i-1}^{s-5}p^{\iota}$ is the sum of the length of the intervals $$\left[t^{'}\sum\limits_{\iota = j-1}^{k-1}p^{\iota}+p^{j-1},\ldots,t^{'}\sum\limits_{\iota = j}^{k-1}p^{\iota}+p^{j}\right],~i\leq j \leq s-4$$ and $t^{'}\sum\limits_{\iota=0}^{s-5}p^{\iota}+\sum\limits_{\iota=s-5}^{k-1}p^{\iota}$ is the length of the interval $$\left[t^{'}\sum\limits_{\iota = s-4}^{k-1}p^{\iota}+p^{s-4},\ldots,(t^{'}+1)\sum\limits_{\iota = 0}^{k-1}p^{\iota}\right].$$
				
				Therefore the position $\alpha+p^{k-1}t^{'}$ corresponding to the component is in the interval $$\left[t^{'}\sum\limits_{\iota = s-4}^{k-1}p^{\iota}+p^{s-4},\ldots,(t^{'}+1)\sum\limits_{\iota = 0}^{k-1}p^{\iota}\right]$$ and its corresponding $p^{(k)}$-Fibonacci number for $0\leq t^{'}\leq t-2$ is 
				
				$			\mathcal{M}\Big(\lambda\big(2(r-1);(x_{(2(s-1),k)},x_{(s-1,k)}+(\alpha+p^{k-1}t^{'}))\big)\Big) $ 
				\begin{align}\label{1st}
		& = \frac{p-1}{2}\big(2(s-2)p^{s-2}+1+2(t^{'}+1)\sum\limits_{j=0}^{s-4}p^{j}\big).
				\end{align}
				When $ t^{'}= t-1$, equation \eqref{proj} can be rewritten as
				\begin{align}
					\alpha+p^{k-1}(t-1) = (t-1)\sum\limits_{\iota=i-1}^{k-1}p^{\iota}+p^{i-1}
				\end{align} 
				The position $\alpha+p^{k-1}(t-1)$ corresponding to the component 
				lies in the interval $$\left[(t-1)\sum\limits_{\iota = i-1}^{k-1}p^{\iota}+p^{i-1},\ldots,(t-1)\sum\limits_{\iota = i}^{k-1}p^{\iota}+p^{i}\right].$$ Hence the corresponding $p^{(k)}$-Fibonacci number for $t^{'}= t-1$ is  
				
				$					\mathcal{M}\Big(\lambda\big(2(r-1);(x_{(2(s-1),k)},x_{(s-1,k)}+(\alpha+p^{k-1}t^{'}))\big)\Big)$
				\begin{align}\label{2nd}
 & = \frac{p-1}{2}\Big(2(s-2)p^{s-2}+1+2t\sum\limits_{j=0}^{s-4}p^{j}-2\sum\limits_{\mu=0}^{s-i-4}p^{\mu}\Big).
				\end{align}
				Next, when $t\leq t^{'}\leq\frac{p-3}{2}$ and $\frac{p+1}{2}\leq t^{'}\leq p-1$, equation \eqref{proj} can be rewritten as
				\begin{align}
					\alpha+p^{k-1}t^{'} = t^{'}\sum\limits_{\iota=i-1}^{k-1}p^{\iota}+p^{i-1}- (t^{'}-(t-1))\sum\limits_{\iota=i-1}^{k-2}p^{\iota}
				\end{align} 
				Observe that $$(p-t^{'})\sum_{\iota=0}^{i-2}p^{\iota}<(t^{'}-(t-1))\sum\limits_{\iota=i-1}^{k-2}p^{\iota}<(p-t^{'})\sum_{\iota=0}^{i-2}p^{\iota}+(t^{'}-1)\sum_{\iota=0}^{s-5}p^{\iota}+\sum_{\iota=s-5}^{k-1}p^{\iota}$$
				where $(p-t^{'})\sum\limits_{\iota=i-1}^{s-5}p^{\iota}$ is the sum of the length of the intervals $$\left[t^{'}\sum\limits_{\iota = j-1}^{k-1}p^{\iota}+p^{j-1},\ldots,t^{'}\sum\limits_{\iota = j}^{k-1}p^{\iota}+p^{j}\right],~1\leq j \leq i-1$$ and $(t^{'}-1)\sum\limits_{\iota=0}^{s-5}p^{\iota}+\sum\limits_{\iota=s-5}^{k-1}p^{\iota}$ is the length of the interval $$\left[(t^{'}-1)\sum\limits_{\iota = s-4}^{k-1}p^{\iota}+p^{s-4},\ldots,t^{'}\sum\limits_{\iota = 0}^{k-1}p^{\iota}\right].$$
				
				Therefore, the position $\alpha+p^{k-1}t^{'}$ corresponding to the component is in the interval $$\left[(t^{'}-1)\sum\limits_{\iota = s-4}^{k-1}p^{\iota}+p^{s-4},\ldots,t^{'}\sum\limits_{\iota = 0}^{k-1}p^{\iota}\right].$$ Hence, for $t\leq t^{'}\leq\frac{p-3}{2}$, the corresponding $p^{(k)}$-Fibonacci number is  
				
				$		\mathcal{M}\Big(\lambda\big(2(r-1);(x_{(2(s-1),k)},x_{(s-1,k)}+(\alpha+p^{k-1}t^{'}))\big)\Big) $
				\begin{align}\label{4thi}
			& = \frac{p-1}{2}\Big(2(s-2)p^{s-2}+1+2t^{'}\sum\limits_{j=0}^{s-4}p^{j}\Big).
				\end{align}
				and for $\frac{p-1}{2}\leq t^{'}\leq p-1$, is
				
				$\mathcal{M}\Big(\lambda\big(2(r-1);(x_{(2(s-1),k)},x_{(s-1,k)}+(\alpha+p^{k-1}t^{'}))\big)\Big)$
				\begin{align}\label{4thii}
					 & = \frac{p-1}{2}\Big(2(s-2)p^{s-2}+1+(2t^{'}-2p)\sum\limits_{j=0}^{s-4}p^{j}\Big).
				\end{align}
				When $t^{'}=\frac{p-1}{2}$, equation \eqref{proj} can be rewritten as
				\begin{align*}
					\alpha+p^{k-1}t^{'} = (t-1)\sum\limits_{\iota=i-1}^{k-2}p^{\iota}+p^{i-1}+ \left(\frac{p-1}{2}\right)p^{k-1}
				\end{align*} 
				Observe that $$\left(\frac{p-1}{2}\right)p^{k-1}\in \left[\left(\frac{p-3}{2}\right)\sum\limits_{\iota = s-4}^{k-1}p^{\iota}+p^{s-4},\ldots,\left(\frac{p-1}{2}\right)\sum\limits_{\iota = s-3}^{k-1}p^{\iota}-1\right].$$ Moreover, the difference between $\left(\frac{p-1}{2}\right)p^{k-1}$ and $\left(\frac{p-1}{2}\right)\sum\limits_{\iota = s-3}^{k-1}p^{\iota}$ is $\left(\frac{p-1}{2}\right)\sum\limits_{\iota = s-3}^{k-2}p^{\iota}$ which is greater than $(t-1)\sum\limits_{\iota=i-1}^{k-2}p^{\iota}+p^{i-1}$.  
				
				Therefore the position $\alpha+p^{k-1}\left(\frac{p-1}{2}\right)$ corresponding to the component is in the interval $$\left[\left(\frac{p-3}{2}\right)\sum\limits_{\iota = s-4}^{k-1}p^{\iota}+p^{s-4},\ldots,\left(\frac{p-1}{2}\right)\sum\limits_{\iota = s-3}^{k-1}p^{\iota}-1\right]$$ and its corresponding $p^{(k)}$-Fibonacci number for $t^{'}=\left(\frac{p-1}{2}\right)$ is  
			\begin{equation}\label{3rd}
					\begin{aligned}
					\mathcal{M}\Big(\lambda\big(2(r-1);(x_{(2(s-1),k)},x_{(s-1,k)}+(\alpha+p^{k-1}t^{'}))\big)\Big) \\& \hspace*{-3cm}= \frac{p-1}{2}\big(2(s-2)p^{s-2}+1+p^{s-3}-1\big).
				\end{aligned}
			\end{equation}
				Substituting the $p^{(k)}$-Fibonacci numbers from equations \eqref{1st}, \eqref{2nd}, \eqref{4thi}, \eqref{4thii}, and \eqref{3rd} into equation \eqref{less l} of Definition \ref{Fibo}, we obtain 
				\begin{align*}
					\mathcal{M}\Big(\lambda\big(2r;(x_{(2s,k)},x_{(s,k)}+l)\big)\Big)\\ &\hspace*{-3cm}=\frac{p-1}{2}\Big(2(s-1)\cdot p^{s-1}+1+2t\sum_{j=0}^{s-3}p^{j}-2\sum_{\mu=0}^{s-i-4}p^{\mu}\Big)
				\end{align*}	
				\item [(ii)] Let $l\in\left[(t-1)\sum\limits_{\iota=s-3}^{k-1}p^{\iota}+p^{s-3},\ldots,t\sum\limits_{\iota=0}^{k-1}p^{\iota}\right]$, $1\leq t\leq \frac{p-3}{2}$. The $p^{(k)}$-Fibonacci numbers of the components are the same as those given in equations \eqref{1st}, \eqref{4thi}, \eqref{4thii}, \eqref{3rd} except for the case $t^{'}=t-1$. The $p^{(k)}$-Fibonacci number of the component corresponding to the position $\alpha+p^{k-1}t^{'}$ with $t^{'}=t-1$ is,
				
				$	\mathcal{M}\Big(\lambda\big(2(r-1);(x_{(2(s-1),k)},x_{(s-1,k)}+(\alpha+p^{k-1}(t-1)))\big)\Big)$
				\begin{align}\label{sp2}
				&= \frac{p-1}{2}\big(2(s-2)p^{s-2}+1+2t\sum\limits_{j=0}^{s-4}p^{j}\big).
				\end{align}
				From equations (\ref{2nd}) and (\ref{sp2}), it follows that the interval $$\left[(t-1)\sum\limits_{\iota=s-4}^{k-1}p^{\iota}+p^{s-4},\ldots,t\sum\limits_{\iota=0}^{k-1}p^{\iota}\right]$$ on the $2(r-1)$-th floor is divided into the two subintervals $$\left[(t-1)\sum\limits_{\iota=s-4}^{k-1}p^{\iota}+p^{s-4},\ldots,(t-1)\sum\limits_{\iota=s-3}^{k-1}p^{\iota}+p^{s-3}-1\right]$$ and $$ \left[(t-1)\sum\limits_{\iota=s-3}^{k-1}p^{\iota}+p^{s-3},\ldots,t\sum\limits_{\iota=0}^{k-1}p^{\iota}\right]$$ for each $1\leq t\leq \frac{p-3}{2}$, since the $p^{(k)}$-Fibonacci number of the component varies at the position $\alpha +p^{k-1}(t-1)$.
				
				Substituting the $p^{(k)}$-Fibonacci numbers from equations \eqref{1st}, \eqref{4thi}, \eqref{4thii}, \eqref{3rd}, and \eqref{sp2} into equation \eqref{less l} of Definition \ref{Fibo}, we obtain 
				$$	\mathcal{M}\Big(\lambda\big(2r;(x_{(2s,k)},x_{(s,k)}+l)\big)\Big)=\frac{p-1}{2}\Big(2(s-1)\cdot p^{s-1}+1+2t\sum\limits_{j=0}^{s-3}p^{j}\Big)$$		
			\end{enumerate} 
			\item[(c)] 
			\begin{enumerate}
				\item [(i)] Let $l\in\left[\frac{p-3}{2}\sum\limits_{\iota=i}^{k-1}p^{\iota}+p^{i},\ldots,\frac{p-3}{2}\sum\limits_{\iota=i+1}^{k-1}p^{\iota}+p^{i+1}-1\right],~0\leq i \leq s-4,~s\neq3$. The $p^{(k)}$-Fibonacci numbers of the components are the same as those in Case \textbf{(b)}(i), by taking $t=\frac{p-3}{2}$. Substituting the corresponding values of the $p^{(k)}$-Fibonacci numbers into equation \eqref{less l} of Definition \ref{Fibo}, we obtain 
				\begin{align*}
					\mathcal{M}\Big(\lambda\big(2r;(x_{(2s,k)},x_{(s,k)}+l)\big)\Big) \\&\hspace*{-3cm}=\frac{p-1}{2}\Big(2(s-1)p^{s-1}+1+(p-1)\sum_{j=0}^{s-3}p^{j}
					-2\sum_{\mu=0}^{s-i-4}p^{\mu}\Big)
				\end{align*} 
				In particular, when $t^{'}=\frac{p-1}{2}$, the position $\alpha+p^{k-1}\left(\frac{p-1}{2}\right)$ corresponding to the component is in the interval $$\left[\frac{p-3}{2}\sum\limits_{\iota = s-4}^{k-1}p^{\iota}+p^{s-4},\ldots,\frac{p-1}{2}\sum\limits_{\iota = s-3}^{k-1}p^{\iota}-1\right]$$ where $\alpha = \frac{p-3}{2}\sum\limits_{\iota=i-1}^{k-2}p^{\iota}+p^{i-1}$ and its corresponding $p^{(k)}$-Fibonacci number is 
				\begin{equation}\label{c0}
				\begin{aligned}
					\mathcal{M}\Bigg(\lambda\Big(2(r-1);\Big(x_{(2(s-1),k)},x_{(s-1,k)}+\Big(\alpha+\Big(\frac{p-1}{2}\Big)p^{k-1}\Big)\Big)\Big)\Bigg)\\ &\hspace*{-6cm}=\frac{p-1}{2}\left(2(s-2)p^{s-2}+1+p^{s-3}-1\right).
				\end{aligned}
				\end{equation}
				\item [(ii)] Let	$l\in\left[\frac{p-3}{2}\sum\limits_{\iota=s-3}^{k-1}p^{\iota}+p^{s-3},\ldots,\frac{p-1}{2}\sum\limits_{\iota=s-2}^{k-1}p^{\iota}-1\right]$. The components have the same $p^{(k)}$-Fibonacci numbers as those in Case \textbf{(b)}(i), except for the case $t^{'}=\frac{p-1}{2}$. Let $l=\frac{p-3}{2}\sum\limits_{\iota=s-3}^{k-1}p^{\iota}+p^{s-3}$. Then by Definition \ref{comp}, the position corresponding to the component is  
				\begin{align*}
					\alpha+p^{k-1}\left(\frac{p-3}{2}\right) = \frac{p-3}{2}\sum\limits_{\iota=s-4}^{k-1}p^{\iota}+p^{s-4} 
				\end{align*}
				where $\alpha = \frac{p-3}{2}\sum\limits_{\iota=s-4}^{k-1}p^{\iota}$.
				
				Hence, the position $\alpha+p^{k-1}\left(\frac{p-1}{2}\right)$ corresponding to the component is in the interval $$\left[\frac{p-3}{2}\sum\limits_{\iota=s-4}^{k-1}p^{\iota}+p^{s-4},\ldots,\frac{p-1}{2}\sum\limits_{\iota=s-3}^{k-1}p^{\iota}-1\right]$$ and its $p^{(k)}$-Fibonacci number is 
				\begin{equation}
				\begin{aligned}\label{ci}
				M\Bigg(\lambda\Big(2(r-1);\Big(x_{(2(s-1),k)},x_{(s-1,k)}+\Big(\alpha+p^{k-1}\Big(\frac{p-1}{2}\Big)\Big)\Big)\Big)\Bigg)\\ & \hspace*{-6cm}= \frac{p-1}{2}\left(2(s-2)p^{s-2}+1-p^{s-3}-1\right).
			\end{aligned}		
				\end{equation}
				Substituting the $p^{(k)}$-Fibonacci numbers from equations \eqref{1st}, \eqref{4thii}, and \eqref{ci} into equation \eqref{less l} of Definition \ref{Fibo}, we obtain $$\mathcal{M}\Big(\lambda\big(2r;(x_{(2s,k)},x_{(s,k)}+l)\big)\Big)=\frac{p-1}{2}\left(2(s-1)\cdot p^{s-1}+1+p^{s-2}-1\right)$$
				\item [(iii)] Let $l\in\left[\frac{p-1}{2}\sum\limits_{\iota=i^{'}}^{k-1}p^{\iota},\ldots,\frac{p-1}{2} \sum\limits_{\iota=i^{'}-1}^{k-1}p^{\iota}-1\right],~2\leq i^{'}\leq s-2$. For each $i^{'}$, the $p^{(k)}$-Fibonacci numbers of the components are the same as those in Case \textbf{(b)}(i), except in the case $t^{'}=\frac{p-1}{2}$. Let $l=\frac{p-1}{2}\sum\limits_{\iota=i^{'}}^{k-1}p^{\iota}$. Then by Definition \ref{comp}, we have
				\begin{align*}
					\alpha+p^{k-1}\left(\frac{p-1}{2}\right) = \frac{p-1}{2}\sum\limits_{\iota=i^{'}-1}^{k-1}p^{\iota} \quad 
				\end{align*}
				where $\alpha = \frac{p-1}{2}\sum\limits_{\iota=i^{'}-1}^{k-2}p^{\iota}$. Hence, the position $\alpha+p^{k-1}\left(\frac{p-1}{2}\right)$ corresponding to the component is in the interval $$\left[\frac{p-1}{2}\sum\limits_{\iota=i^{'}-1}^{k-1}p^{\iota},\ldots,\frac{p-1}{2}\sum\limits_{\iota=i^{'}-2}^{k-1}p^{\iota}-1\right].$$
				The corresponding $p^{(k)}$-Fibonacci number of the component is  
				
				$M\Bigg(\lambda\Big(2(r-1);\Big(x_{(2(s-1),k)},x_{(s-1,k)}+\Big(\alpha+p^{k-1}\Big(\frac{p+1}{2}\Big)\Big)\Big)\Bigg)$
				\begin{align}\label{ciii}
					& = \frac{p-1}{2}\Big(2(s-2)\cdot p^{s-2}-1-p^{s-3}-2\sum\limits_{\theta=0}^{s-4}p^{\theta}+\sum\limits_{\mu=s-i^{'}-1}^{s-3}p^{\mu}-1\Big)
				\end{align}
				Substituting the $p^{(k)}$-Fibonacci numbers from equations \eqref{1st}, \eqref{4thii}, and \eqref{ciii}, into equation \eqref{less l} of Definition \ref{Fibo}, we obtain
				
				$\mathcal{M}\Big(\lambda\big(2r;(x_{(2s,k)},x_{(s,k)}+l)\big)\Big)$
				\begin{align*}
					&=\frac{p-1}{2}\Big(2(s-1)\cdot p^{s-1}-1-p^{s-2}-2\sum\limits_{\theta=0}^{s-3}p^{\theta}+\sum\limits_{\mu=s-i^{'}-1}^{s-2}p^{\mu}-1\Big).
				\end{align*}
				
				Next for $i^{'}=1$, let $l\in\left[\frac{p-1}{2}\sum\limits_{\iota=1}^{k-1}p^{\iota},\ldots,\frac{p-1}{2} \sum\limits_{\iota=0}^{k-1}p^{\iota}-1\right]$. From equation \eqref{less l}, we obtain
				
				$\mathcal{M}\Big(\lambda\big(2r;(x_{(2s,k)},x_{(s,k)}+l)\big)\Big)$ \begin{align*}
					& =\frac{p-1}{2}\Big(2(s-1)\cdot p^{s-1}-1-p^{s-2}-2\sum\limits_{\theta=0}^{s-3}p^{\theta}+\sum\limits_{\mu=s-2}^{s-2}p^{\mu}-1\Big).
				\end{align*}
				Here, $\mathcal{M}\Big(\lambda\big(2(r-1);(x_{(2(s-1),k)},x_{(s-1,k)}+(\alpha+p^{k-1}t^{'}))\big)\Big)$ coincides with Case \textbf{(b)}(i) except for the case $t^{'}=\frac{p-1}{2}$. The corresponding $p^{(k)}$-Fibonacci number of the component is
				
				$\mathcal{M}\Bigg(\lambda\Big(2(r-1);\Big(x_{(2(s-1),k)},x_{(s-1,k)}+\Big(\alpha+p^{k-1}\Big(\frac{p+1}{2}\Big)\Big)\Big)\Bigg)$
				\begin{align}						
					 &=\frac{p-1}{2}\Big( 2(s-2)p^{s-2}-1-p^{s-3}-2\sum\limits_{\theta=0}^{s-4}p^{\theta}-1\Big).\label{cv}
				\end{align} 
				From equations \eqref{c0}, \eqref{ci} and \eqref{ciii} it follows that the interval $$\left[\frac{p-3}{2}\sum\limits_{\iota=s-4}^{k-1}p^{\iota}+p^{s-4},\ldots,\frac{p-1}{2}\sum\limits_{\iota=s-3}^{k-1}p^{\iota}-1\right]$$ in the $2(r-1)$-th floor is divided into three subintervals: $$\left[\frac{p-3}{2}\sum\limits_{\iota=s-4}^{k-1}p^{\iota}+p^{s-4},\ldots,\frac{p-3}{2}\sum\limits_{\iota=s-3}^{k-1}p^{\iota}+p^{s-3}-1\right],$$ $$\left[\frac{p-3}{2}\sum\limits_{\iota=s-3}^{k-1}p^{\iota}+p^{s-3},\ldots,\frac{p-1}{2}\sum\limits_{\iota=s-2}^{k-1}p^{\iota}-1\right], $$ and $$\left[\frac{p-1}{2}\sum\limits_{\iota=s-2}^{k-1}p^{\iota},\ldots,\frac{p-1}{2}\sum\limits_{\iota=s-3}^{k-1}p^{\iota}-1\right]. $$ The subdivision occurs, since the $p^{(k)}$-Fibonacci number of the components varies at the position  $\alpha+p^{k-1}t^{'}$ when $t^{'}=\frac{p-1}{2}$. 
				\item [(iv)] 	For $l=\frac{p-1}{2}\sum\limits_{\iota=0}^{k-1}p^{\iota}$, the components have the same $p^{(k)}$-Fibonacci numbers as in the previous case with $i^{'}=1$, since the respective components are identical.  
Substituting the corresponding $p^{(k)}$-Fibonacci numbers of the components into equation \eqref{greater l} of Definition \ref{Fibo}, we obtain
				\begin{align*}
					\mathcal{M}\Big(\lambda\big(2r;(x_{(2s,k)},x_{(s,k)}+l)\big)\Big)\\&\hspace*{-1cm}=\frac{p-1}{2}\Big(2(s-1)\cdot p^{s-1}-1-p^{s-2}-2\sum\limits_{\theta=0}^{s-3}p^{\theta}-1\Big)
				\end{align*}			
			\end{enumerate}
			\item [(d)] The argument follows as in Case \textbf{(b)}, except that we substitute the $p^{(k)}$-Fibonacci numbers of the components into equation \eqref{greater l}.
			
		\end{description}	
	\end{proof}
	\begin{theorem}\label{greater}
		For $2\leq k\leq r-1$ and $s\geq k+2$ $(\text{where}~r-k=s)$, the $p^{(k)}$-Fibonacci numbers at the vertex $\lambda\big(2r;(x_{(2s,k)},x_{(s,k)}+l)\big)$, where $0\leq l <p^{k}$, is given by 
		\begin{description}
			\item[(a)] If $l=0$, then \begin{align*}
			                            \mathcal{M}\Big(\lambda\big(2r;(x_{(2s,k)},x_{(s,k)}+l)\big)\Big)&=\frac{p^{s-2-k}(p-1)}{2}(2(s-1)p^{k+1}-2(s-k)+3)
			                          \end{align*}
			\item[(b)] For $1\leq t\leq \frac{p-3}{2}$, 
			\begin{enumerate}
				\item [(i)] If $l\in\left[(t-1)\sum\limits_{\iota=i}^{k-1}p^{\iota}+p^{i},\ldots,(t-1)\sum\limits_{\iota=i+1}^{k-1}p^{\iota}+p^{i+1}-1\right],~0\leq i \leq k-2$, then 
				\begin{align*} \allowdisplaybreaks
					\\ \\ \mathcal{M}\Big(\lambda\big(2r;(x_{(2s,k)},x_{(s,k)}+l)\big)\Big) \\ &\hspace*{-4.5cm}=	\frac{p^{s-2-k}(p-1)}{2}\Big(2(s-1)p^{k+1}-2(s-k)+3+2t\sum_{j=0}^{k-1}p^{j}-2\sum_{\mu=0}^{k-i-2}p^{\mu}\Big)
				\end{align*}
				\item [(ii)]If $l\in\left[tp^{k-1},\ldots,t\sum\limits_{\iota=0}^{k-1}p^{\iota}\right]$, then
				\begin{align*}
				 \mathcal{M}\Big(\lambda\big(2r;(x_{(2s,k)},x_{(s,k)}+l)\big)\Big) \\	&\hspace*{-3cm}=\frac{p^{s-2-k}(p-1)}{2}\Big(2(s-1)p^{k+1}-2(s-k)+3+2t\sum_{j=0}^{k-1}p^{j}\Big)
				\end{align*}				
			\end{enumerate}
			\item[(c)] For $t=\frac{p-1}{2},$
			\begin{enumerate}
				\item [(i)] If $l\in\left[\frac{p-3}{2}\sum\limits_{\iota=i}^{k-1}p^{\iota}+p^{i},\ldots,\frac{p-3}{2}\sum\limits_{\iota=i+1}^{k-1}p^{\iota}+p^{i+1}-1\right],~0\leq i \leq k-2$ then 
				\begin{align*}
				\mathcal{M}\Big(\lambda\big(2r;(x_{(2s,k)},x_{(s,k)}+l)\big)\Big) \\	&\hspace*{-5cm}=	\frac{p^{s-2-k}(p-1)}{2}\Big(2(s-1)p^{k+1}-2(s-k)+3+(p-1)\sum_{j=0}^{k-1}p^{j}-2\sum_{\mu=0}^{k-i-2}p^{\mu}\Big)
				\end{align*}
				\item [(ii)] If $l\in\left[\frac{p-1}{2}\sum\limits_{\iota=i^{'}}^{k-1}p^{\iota},\ldots,\frac{p-1}{2} \sum\limits_{\iota=i^{'}-1}^{k-1}p^{\iota}-1\right],~1\leq i^{'}\leq k-1$ then 
				\begin{align*}
		\mathcal{M}\Big(\lambda\big(2r;(x_{(2s,k)},x_{(s,k)}+l)\big)\Big)&= \frac{p^{s-2-k}(p-1)}{2}\Big(2(s-1)p^{k+1}-2(s-k)+3\\ &  \quad-p^{k}-2\sum_{\theta = 1}^{k-1}p^{\theta}+2\sum_{\mu =k-i^{'}+1}^{k}p^{\mu}-1\Big)
				\end{align*}
				\item [(iii)] If $l=\frac{p-1}{2} \sum\limits_{\iota=0}^{k-1}p^{\iota}$, then 
				\begin{align*}
					\mathcal{M}\Big(\lambda\big(2r;(x_{(2s,k)},x_{(s,k)}+l)\big)\Big) \\ &\hspace*{-4cm}=\frac{p^{s-2-k}(p-1)}{2}\Big(2(s-1)p^{k+1}-2(s-k)+3-p^{k}-2\sum_{\theta = 1}^{k-1}p^{\theta}-1\Big)
				\end{align*} 
			\end{enumerate}
			\item[(d)] For $\frac{p+1}{2}\leq t\leq p-1,$ \allowdisplaybreaks
			\begin{enumerate}
				\item [(i)] 	If $l\in\left[(t-1)\sum\limits_{\iota=i}^{k-1}p^{\iota}+p^{i},\ldots,(t-1)\sum\limits_{\iota=i+1}^{k-1}p^{\iota}+p^{i+1}-1\right],~0\leq i \leq k-2$ then
				\begin{align*}					 	 \mathcal{M}\Big(\lambda\big(2r;(x_{(2s,k)},x_{(s,k)}+l)\big)\Big)\\&\hspace*{-5.2cm}=\frac{p^{s-2-k}(p-1)}{2}\Big(2(s-1)p^{k+1}-2(s-k)+3+2(t-p)\sum_{j=0}^{k-1}p^{j}-2\sum_{\mu=0}^{k-i-2}p^{\mu}\Big)
				\end{align*}
				\item [(ii)]If $l\in\left[tp^{k-1},\ldots,t\sum\limits_{\iota=0}^{k-1}p^{\iota}\right]$, then
			\begin{align*}
					  \mathcal{M}\Big(\lambda\big(2r;(x_{(2s,k)},x_{(s,k)}+l)\big)\Big)\\&\hspace*{-4cm}=\frac{p^{s-2-k}(p-1)}{2}\Big(2(s-1)p^{k+1}-2(s-k)+3+(2t-2p)\sum_{j=0}^{k-1}p^{j}
					\Big)
				\end{align*} 
			\end{enumerate}
		\end{description}
	\end{theorem}
	\begin{proof}
		We prove the result by mathematical induction, by repeatedly applying equations (\ref{l0}), (\ref{less l}) and (\ref{greater l}).
		
	\end{proof}
	
	\begin{example}
		Let us discuss the paths starting at the vertex $\lambda\big(10;(x_{(6,2)},x_{(3,2)}+9)\big)$ and ending at the vertices $\lambda\big(1;(p-1,i)\big)$, where $0\leq i \leq p-2$,  using the $p$-Bratteli diagram with $p=5$. Here $r=5$, $s=3$, $x_{(6,2)} = 575$, $x_{(3,2)}=275$ and $l=9$. Furthermore, we compute the $p^{(k)}$-Fibonacci numbers at the vertices along these paths.

		The paths from the $10^{th}$ floor to the $6^{th}$ floor are illustrated in Figure \ref{ex}.
		
		\begin{figure}[!ht]
			\centering
			\includegraphics[width=12cm,height=8.8cm]{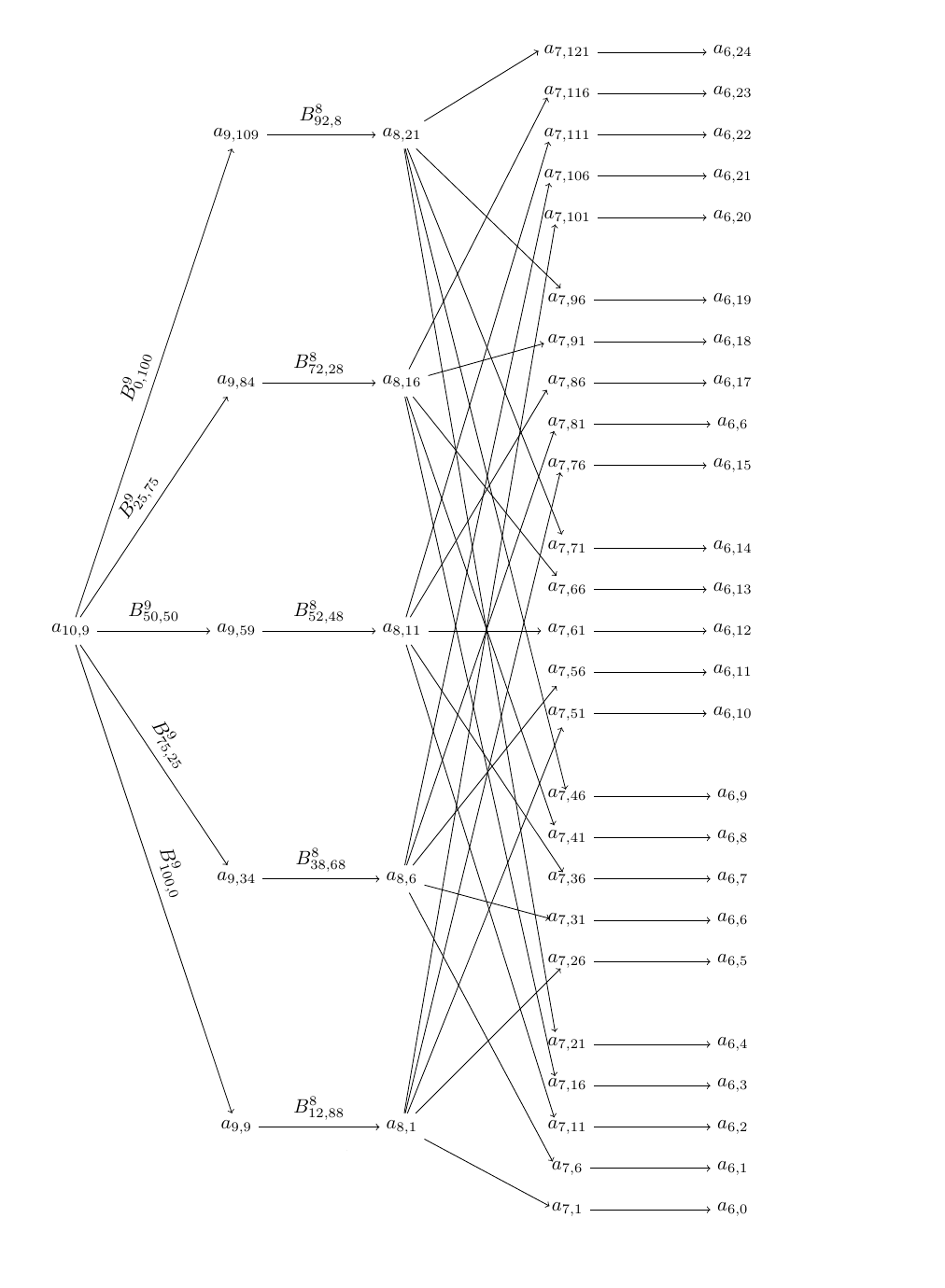}
			\caption{}
			\label{ex}
		\end{figure}  
		
		Here $a_{r,j}$ denotes the vertex in the $j^{th}$ position at the $r^{th}$ floor and 
		\begin{itemize}
			\item $a_{10,9}=\lambda\big(10;(575,275+9)\big)$. \item $a_{9,9}=\lambda\big(9;(475,175+9)\big)$, $a_{9,34}=\lambda\big(9;(475,175+34)\big)$, $a_{9,59}=\lambda\big(9;(475,175+59)\big)$, $a_{9,84}=\lambda\big(9;(475,175+84)\big)$, $a_{9,109}=\lambda\big(9;(475,175+109)\big)$.
			\item $a_{8,1}=\lambda\big(8;(375,175+1)\big)$, $a_{8,1}=\lambda\big(8;(375,175+6)\big)$, $a_{8,1}=\lambda\big(8;(375,175+11)\big)$, $a_{8,1}=\lambda\big(8;(375,175+16)\big)$, $a_{8,1}=\lambda\big(8;(375,175+21)\big)$. 
			\item $a_{7,l^{'}}=\lambda\big(7;(275,75+l^{'})\big)$, where $l^{'} = 1,~6,~11,~16,~21,~26,~31,~36,~41,~46,~51,~56,~61,$ $66,~71,~76$, $~81,~86,~91,~96,~101,~106,~111,~116,~121$. 
			\item $a_{6,l^{''}}=\lambda\big(6;(175,75+l^{''})\big)$, where $0\leq l^{''}\leq 24$.
			
		\end{itemize}
		
		The paths from a particular vertex $\lambda\big(6;(175,75+5)\big)$,  at the $6^{th}$ floor to the vertices at the $1^{st}$ floor are shown in Figure \ref{exm}:
		\begin{figure}[!ht]
			\centering
			\includegraphics[width=8cm,height=10cm]{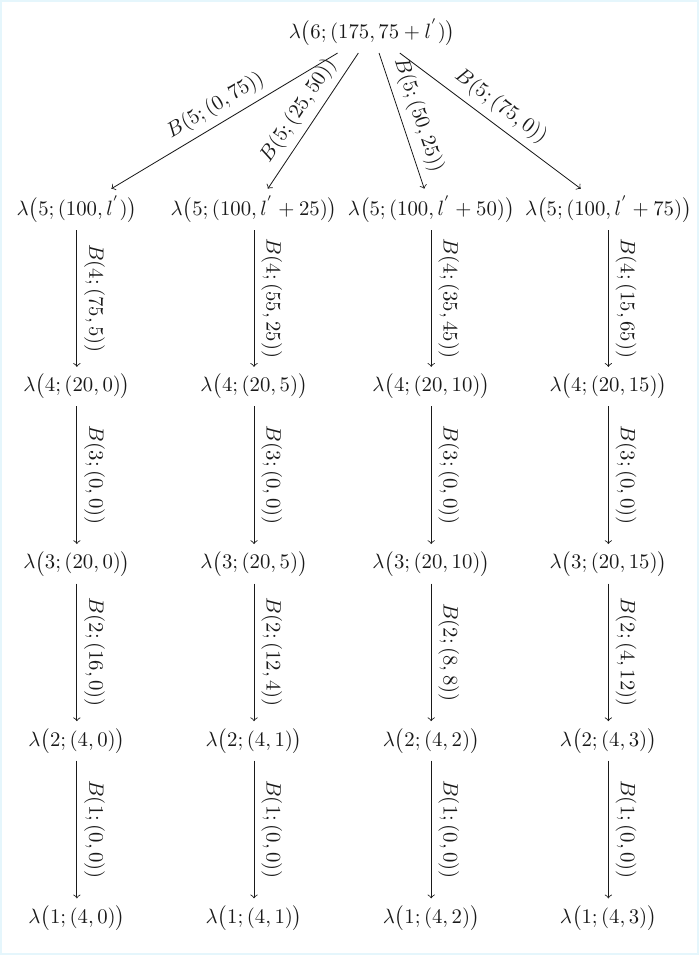}
			\caption{}
			\label{exm}
		\end{figure}  
		
		Similarly, we construct the paths from $6^{th}$ floor to $1^{st}$ floor
		for other vertices as well.
		We define the vertices are defined as in Section \ref{vertex}, and we obtain the edges by removing the blocks as described in Section \ref{edge}.
		
		The $p^{(k)}$-Fibonacci numbers at these vertices are as follows:
		\begin{itemize}
			\item $\mathcal{M}\big(\lambda\big(6;(175,75+l^{''})\big)\big) = 2$ for all $0\leq l^{''}\leq 24$ (By Theorem \ref{base}).
			\item $\mathcal{M}\big(\lambda\big(8;(375,175+1)\big)\big) = 22$, $\mathcal{M}\big(\lambda\big(8;(375,175+6)\big)\big) = 22$, $\mathcal{M}\big(\lambda\big(8;(375,175+11)\big)\big) = 22$, $\mathcal{M}\big(\lambda\big(8;(375,175+16)\big)\big) = 18$, $\mathcal{M}\big(\lambda\big(8;(375,175+21)\big)\big) = 18$ (By Theorem \ref{base}).
			\item $\mathcal{M}\big(\lambda\big(10;(575,275+9)\big)\big) = 210$ (By Theorem \ref{less}).
		\end{itemize}
	\end{example}
	
	The following theorems justify this name \textbf{($p^{(k)}$-Fibonacci numbers)}.
	\begin{theorem}\label{rr}
		For each $k\geq 0$ and $s\geq k+2$ $(\text{where}~r-k=s)$, the  $p^{(k)}$-Fibonacci numbers satisfy the following recurrence relation: 
		\begin{enumerate}
			\item When $k=0$,
			
			$\mathcal{M}\Big(\lambda\big(2(r+2);(x_{(2(s+2),0)},x_{(s+2,0)})\big)\Big)$
			\begin{align*}
				&=b_{s}+p\mathcal{M}\Big(\lambda\big(2r;(x_{(2s,0)},x_{(s,0)})\big)\Big)+(p-1)\mathcal{M}\Big(\lambda\big(2(r+1);(x_{(2(s+1),0)},x_{(s+1,0)})\big)\Big)
			\end{align*}
 where $b_{s}=2p^{s-1}(p-1)\left(\frac{p^{2}-1}{2}\right)$.\\ \\
			\item When $k=1,$
			\begin{enumerate}
				\item For $0\leq t <\frac{p-1}{2}$
				
				$\mathcal{M}\Big(\lambda\big(2(r+2);(x_{(2(s+2),1)},x_{(s+2,1)}+j^{1}_{t})\big)\Big)$
				\begin{align*}
					&=b_{s}+\sum\limits_{t_{2}=0}^{p-1}\mathcal{M}\Big(\lambda\big(2r;(x_{(2s,1)},x_{(s,1)}+j^{1}_{t_{2}})\big)\Big)+ \\ & \quad \sum\limits_{t_{1}=1}^{p-1}\mathcal{M}\Big(\lambda\big(2(r+1);(x_{(2(s+1),1)},x_{(s+1,1)}+j^{1}_{t_{1}})\big)\Big),
				\end{align*}
				where $b_{s}=p^{s-1}(p-1)\left(p^{2}+p+t\right)$. \label{k1just}
				\item For $\frac{p-1}{2}\leq t\leq p-1$
				
				$\mathcal{M}\Big(\lambda\big(2(r+2);(x_{(2(s+2),1)},x_{(s+2,1)}+j^{1}_{t})\big)\Big)$
				\begin{align*}
				&=b_{s}+\sum_{t_{2}=0}^{p-1}\mathcal{M}\Big(\lambda\big(2r;(x_{(2s,1)},x_{(s,1)}+j^{1}_{t_{2}})\big)\Big)+ \\ & \quad\sum_{t_{1}=1}^{p-1}\mathcal{M}\Big(\lambda\big(2(r+1);(x_{(2(s+1),1)},x_{(s+1,1)}+j^{1}_{t_{1}})\big)\Big),
				\end{align*}
			 where $b_{s}=p^{s-1}(p-1)\left(p^{2}+t\right)$. \label{k2just}\\ \\
			\end{enumerate}
			\item When $k\geq 2,$
			\begin{enumerate}
				\item For $j^{k}_{t-1}<l <j^{k}_{t},~1\leq t\leq \frac{p-1}{2}$ and $l =j^{k}_{t},~0\leq t< \frac{p-1}{2}$
				
				$\mathcal{M}\Big(\lambda\big(2(r+2);(x_{(2(s+2),k)},x_{(s+2,k)}+l)\big)\Big)$
				\begin{align*}
					&=b_{s}+\sum\limits_{t_{2}=0}^{p-1}\mathcal{M}\Big(\lambda\big(2r;(x_{(2s,k)},x_{(s,k)}+(\alpha^{'}+p^{k-1}t_{2}))\big)\Big)+ \\ &  \quad\sum\limits_{t_{1}=1}^{p-1}\mathcal{M}\Big(\lambda\big(2(r+1);(x_{(2(s+1),k)},x_{(s+1,k)}+(\alpha+p^{k-1}t_{1}))\big)\Big),
				\end{align*}
				where $b_{s}=p^{s-1}(p-1)\left(p^{2}+p+t\right)$. 
				
				The vertices $\lambda\big(2(r+1);(x_{(2(s+1),k)},x_{(s+1,k)}+(\alpha+p^{k-1}t_{1}))\big)$, where $0\leq t_{1}\leq p-1$ are the components of $\lambda\big(2(r+2);(x_{(2(s+2),k)},x_{(s+2,k)}+l)\big)$ and the vertices $\lambda\big(2r;(x_{(2s,k)},x_{(s,k)}+(\alpha^{'}+p^{k-1}t_{2}))\big)$, where $0\leq t_{2}\leq p-1$ are the components of  $\lambda\big(2(r+1);(x_{(2(s+1),k)},x_{(s+1,k)}+\alpha)\big)$.\label{k11just}
				\item For $j^{k}_{t-1}\leq  l \leq j^{k}_{t}$, $\frac{p+1}{2}\leq t \leq p-1$ 
				\begin{align*}
					\mathcal{M}\Big(\lambda\big(2(r+2);(x_{(2(s+2),k)},x_{(s+2,k)}+l)\big)\Big)
					\\&\hspace*{-4cm}=b_{s}+\sum\limits_{t_{2}=0}^{p-1}\mathcal{M}\Big(\lambda\big(2r;(x_{(2s,k)},x_{(s,k)}+(\alpha^{'}+p^{k-1}t_{2}))\big)\Big)+ \\ & \hspace*{-4cm}\quad \sum\limits_{t_{1}=1}^{p-1}\mathcal{M}\Big(\lambda\big(2(r+1);(x_{(2(s+1),k)},x_{(s+1,k)}+(\alpha+p^{k-1}t_{1}))\big)\Big),
				\end{align*}
					where $b_{s}=p^{s-1}(p-1)\left(p^{2}+t\right)$. 
					
					The vertices $\lambda\big(2(r+1);(x_{(2(s+1),k)},x_{(s+1,k)}+(\alpha+p^{k-1}t_{1}))\big)$, where $0\leq t_{1}\leq p-1$ are the components of $\lambda\big(2(r+2);(x_{(2(s+2),k)},x_{(s+2,k)}+l)\big)$ and the vertices $	\lambda\big(2r;(x_{(2s,k)},x_{(s,k)}+(\alpha^{'}+p^{k-1}t_{2}))\big)$, where $0\leq t_{2}\leq p-1$ are the components of  $\lambda\big(2(r+1);(x_{(2(s+1),k)},x_{(s+1,k)}+\alpha)\big)$.		\label{k22just}	
			\end{enumerate}
		\end{enumerate}
	\end{theorem}
	\begin{proof}
		\begin{enumerate}
			\item From equation \eqref{k0}), we have 
			
			$\mathcal{M}\big(\lambda\big(2(r+2);(x_{(2(s+2),0)},x_{(s+2,0)})\big)\big)$
			\begin{align*}
				&= 	p\mathcal{M}\Big(\lambda\big(2(r+1);(x_{(2(s+1),0)},x_{(s+1,0)})\big)\Big)+2p^{s}(p-1)\left(\frac{p-1}{2}\right) \\ &=\mathcal{M}\Big(\lambda\big(2(r+1);(x_{(2(s+1),0)},x_{(s+1,0)})\big)\Big)\\ & \quad+(p-1)\mathcal{M}\Big(\lambda\big(2(r+1);(x_{(2(s+1),0)},x_{(s+1,0)})\big)\Big)+2p^{s}(p-1)\left(\frac{p-1}{2}\right)\\ &=b_{s}+ p\mathcal{M}\Big(\lambda\big(2r;(x_{(2s,0)},x_{(s,0)})\big)\Big)+ (p-1)\mathcal{M}\Big(\lambda\big(2(r+1);(x_{(2(s+1),0)},x_{(s+1,0)})\big)\Big)
			\end{align*}
			where $b_{s}=2p^{s-1}(p-1)\left(\frac{p^{2}-1}{2}\right)$.
			
			Proof of (\ref{k1just}), (\ref{k2just}), (\ref{k11just}) and (\ref{k22just}) follow from equations \eqref{l0}, \eqref{less l} and \eqref{greater l} of Definition \ref{Fibo}. 
		\end{enumerate}
	\end{proof}
	
	\section{The generating functions}
	
	Finally, we conclude the study by giving the generating function for the sequence of $p^{(k)}$-Fibonacci numbers $\big(\mathcal{M}\Big(\lambda\big(2r;(x_{(2s,k)},x_{(s,k)}+l)\big)\Big)\big)_{s=k+2}^{\infty}$
	for each $k\geq 0$ and $0\leq l<p^{k}$, and we explain why these numbers are called $p^{(k)}$-Fibonacci numbers.
	\begin{definition}
		For $k\geq 0$, we define the generating function for the sequence of $p^{(k)}$-Fibonacci numbers $(a_{s})_{s=k+2}^{\infty}$ by $$F(x) = \sum_{s=k+2}^{\infty} a_{s}x^{s-(k+2)}.$$
	\end{definition}
	\begin{theorem}\label{gf k0}
		The generating function for the sequence of $p^{(0)}$-Fibonacci numbers $a_{s}=\mathcal{M}\Big(\lambda\big(2r;(x_{(2s,0)},x_{(s,0)})\big)\Big)$, $s\geq 2$ is $$F(x)=\frac{p-1}{2}\left(\frac{2p-2px}{(1-px)^{2}}-\frac{1}{1-px}\right),~|x|<\frac{1}{p}.$$
	\end{theorem}
	\begin{theorem}\label{gf k1}
		The generating function for the sequence of $p^{(1)}$-Fibonacci numbers $a_{s,t}=\mathcal{M}\Big(\lambda\big(2r;(x_{(2s,1)},x_{(s,1)}+j^{1}_{t})\big)\Big)$
		is
		\begin{align*}
			F_{t}(x)&=\frac{p-1}{2}\left(\frac{2p^{2}+2t-1}{1-px}+\frac{2p^{2}-2px}{(1-px)^{2}}\right),~|x|<\frac{1}{p},~0\leq t<\frac{p-1}{2}, \\F_{t}(x)&=\frac{p-1}{2}\left(\frac{2p^{2}+2t-2p-1}{1-px}+\frac{2p^{2}-2px}{(1-px)^{2}}\right),~|x|<\frac{1}{p},~\frac{p-1}{2}\leq t\leq p-1.
		\end{align*}
	\end{theorem}
	\begin{theorem}\label{gf k2}
		The generating function for the sequence of $p^{(k)}$-Fibonacci numbers $a_{s}=\mathcal{M}\Big(\lambda\big(2r;(x_{(2s,k)},x_{(s,k)}+l)\big)\Big)$, defined for each interval in Theorem \ref{greater}, is given below:
		\begin{enumerate}
			\item If $a_{s}=\frac{p^{s-2-k}(p-1)}{2}(2(s-1)p^{k+1}-2(s-k)+3)$ then $$F(x)=\frac{p-1}{2}\left(\frac{2kp^{k+1}-1}{1-px}+\frac{2p^{k+1}-2px}{(1-px)^{2}}\right),~|x|<\frac{1}{p}.$$
			\item 
			\begin{enumerate}
				\item If $a_{s}=b_{s,t,i}$
				\begin{align*}
					&=\frac{p^{s-2-k}(p-1)}{2}\Big(2(s-1)p^{k+1}-2(s-k)+3+2t\sum\limits_{j=0}^{k-1}p^{j}
					-2\sum\limits_{\mu=0}^{k-i-2}p^{\mu}\Big)
				\end{align*}
				  then 
				\begin{align*}
					F_{t,i}(x)&=\frac{p-1}{2}\Bigg(\Big(2t\sum_{j=0}^{k-1}p^{j}-2\sum_{\mu=0}^{k-i-2}p^{\mu}+2kp^{k+1}-1\Big)\frac{1}{1-px}+\frac{2p^{k+1}-2px}{(1-px)^{2}}\Bigg)
				\end{align*}
				where $|x|<\frac{1}{p},~1\leq t\leq \frac{p-1}{2}$ and $0\leq i \leq k-2$.
				\item 	If $a_{s}=b_{s,t,i}$
				\begin{align*}
				&=\frac{p^{s-2-k}(p-1)}{2}\Bigg(2(s-1)p^{k+1}-2(s-k)+3+2(t-p)\sum\limits_{j=0}^{k-1}p^{j}
				-2\sum\limits_{\mu=0}^{k-i-2}p^{\mu}\Bigg)
				\end{align*} 				
				then $F_{t,i}(x)$ is
				\begin{align*}
					\frac{p-1}{2}\Bigg(\Big(2(t-p)\sum_{j=0}^{k-1}p^{j}-2\sum_{\mu=0}^{k-i-2}p^{\mu}+2kp^{k+1}-1\Big)\frac{1}{1-px}+\frac{2(p^{k+1}-px)}{(1-px)^{2}}\Bigg)
				\end{align*}
				where $|x|<\frac{1}{p},~\frac{p+1}{2}\leq t\leq p-1$ and $0\leq i \leq k-2$.
			\end{enumerate}
			\item \begin{enumerate}
				\item If $a_{s}=b_{s,t}=\frac{p^{s-2-k}(p-1)}{2}\left(2(s-1)p^{k+1}-2(s-k)+3+2t\sum\limits_{j=0}^{k-1}p^{j}\right)$  then
				\begin{align*}
					F_{t}(x)&=\frac{p-1}{2}\Bigg(\Big(2t\sum\limits_{j=0}^{k-1}p^{j}+2kp^{k+1}-1\Big)\frac{1}{1-px}+\frac{2p^{k+1}-2px}{(1-px)^{2}}\Bigg),~|x|<\frac{1}{p}  
				\end{align*}
				where $1\leq t\leq \frac{p-3}{2}$.
				\item If $a_{s}=b_{s,t}=\frac{p^{s-2-k}(p-1)}{2}\left(2(s-1)p^{k+1}-2(s-k)+3+(2t-2p)\sum\limits_{j=0}^{k-1}p^{j}
				\right)
				$, then
				\begin{align*}
					F_{t}(x)&=\frac{p-1}{2}\Bigg(\Big((2t-2p
					)\sum\limits_{j=0}^{k-1}p^{j}+2kp^{k+1}-1\Big)\frac{1}{1-px}+\frac{2p^{k+1}-2px}{(1-px)^{2}}\Bigg),			
					\end{align*} 
				where $|x|<\frac{1}{p}$ and $\frac{p+1}{2}\leq t\leq p-1$.
			\end{enumerate} 
			\item If $a_{s}=b_{s,i^{'}}$
			\begin{align*}
				&=\frac{p^{s-2-k}(p-1)}{2}\Big(2(s-1) p^{k+1}-2(s-k)+3-p^{k}-2\sum\limits_{\theta = 1}^{k-1}p^{\theta}+2\sum\limits_{\mu =k-i^{'}+1}^{k}p^{\mu}-1\Big),
			\end{align*}
			 then $F_{i^{'}}(x)$ is
			 \begin{align*}
			 	&=\frac{p-1}{2}\Bigg(\Big(2kp^{k+1}-p^{k}-2\sum\limits_{\theta=1}^{k-1}p^{\theta}+2\sum_{\mu=k-i^{'}+1}^{k}p^{\mu}-2\Big)\frac{1}{1-px}+\frac{2(p^{k+1}-px)}{(1-px)^{2}}\Bigg)
			 \end{align*} 
			where $1\leq i^{'}\leq k-1$ and $|x|<\frac{1}{p}$.
			\item If  $a_{s}= \frac{p^{s-2-k}(p-1)}{2}\left(2(s-1)p^{k+1}-2(s-k)+3-p^{k}-2\sum\limits_{\theta = 1}^{k-1}p^{\theta}-1\right)$ then $$F(x)=\frac{p-1}{2}\Bigg(\Big(2kp^{k+1}-2-p^{k}-2\sum_{\theta=1}^{k-1}p^{\theta}\Big)\frac{1}{1-px}+\frac{2p^{k+1}-2px}{(1-px)^{2}}\Bigg),~|x|<\frac{1}{p}$$
		\end{enumerate}
	\end{theorem}
	\begin{remark}
		The proofs of Theorems \ref{gf k0}, \ref{gf k1} and \ref{gf k2} follow from straightforward calculation, which can be carried out either manually or using mathematical software such as Maple or Sage. We omit the details, as the computation is standard.
	\end{remark}
	
	\section{Acknowledgments}
	The third author acknowledges financial support from the Council of Scientific and Industrial Research (CSIR), India, in the form of CSIR-SRF (File No: 09/0115(16295)/2023-EMR-I) for conducting this research.

	\hrule
	\vspace{0.5em}
	
	2020 \textit{Mathematics Subject Classification:} Primary 05E10; 	Secondary 05A15, 05E16.
	
	\textit{Keywords}: Bratteli diagram, hook partition, inversion, descent, Fibonacci number.
	\vspace{0.5em}
	\hrule
	\vspace{0.5em}
	
	(Concerned with the sequences \seqnum{A391520}.) 
	\vspace{0.5em}
	\hrule
\end{document}